\documentclass[11pt,letterpaper]{amsart}
\usepackage{fullpage,url,amssymb,mathtools,amsmath,verbatim}
\makeatletter
\let\@@pmod\pmod
\DeclareRobustCommand{\pmod}{\@ifstar\@pmods\@@pmod}
\def\@pmods#1{\mkern4mu({\operator@font mod}\mkern 6mu#1)}
\makeatother

\newcommand{\Q}{\mathbb Q}
\newcommand{\C}{\mathbb C}
\newcommand{\R}{\mathbb R}

\newcommand{\A}{\mathbb A}
\newcommand{\p}{\mathfrak p}
\newcommand{\q}{\mathfrak q}
\newcommand{\m}{\mathfrak m}

\newcommand{\n}{{\mathfrak n}}

\renewcommand{\o}{\mathfrak o}

\renewcommand{\a}{\mathfrak a}

\renewcommand{\c}{\mathfrak c}
\newcommand{\idm}{\mathsf m}
\newcommand{\nid}{\mathsf n}
\newcommand{\qid}{\mathsf q}
\newcommand{\aid}{\mathsf a}
\newcommand{\pid}{\mathsf p}
\newcommand{\uid}{\mathsf u}
\newcommand{\vid}{\mathsf v}
\newcommand{\rid}{\mathsf r}
\newcommand{\W}{\mathcal W}
\renewcommand{\AA}{\mathcal A}
\newcommand{\PP}{\mathbb P}
\renewcommand{\P}{\operatorname{P}}

\newcommand{\B}{\operatorname{B}}
\newcommand{\T}{\operatorname{T}}
\newcommand{\GL}{\operatorname{GL}}
\newcommand{\SL}{\operatorname{SL}}

\newcommand{\U}{\operatorname{U}}

\newcommand{\ord}{\operatorname{ord}}

\newcommand{\Ind}{\operatorname{Ind}}
\newcommand{\gr}{Gr\"{o}\ss{}encharakter}
\newcommand{\vol}{\operatorname{vol}}
\newcommand{\bigboxplus}{
  \mathop{
    \vphantom{\bigoplus}
    \mathchoice
      {\vcenter{\hbox{\resizebox{\widthof{$\displaystyle\bigoplus$}}{!}{$\boxplus$}}}}
      {\vcenter{\hbox{\resizebox{\widthof{$\bigoplus$}}{!}{$\boxplus$}}}}
      {\vcenter{\hbox{\resizebox{\widthof{$\scriptstyle\oplus$}}{!}{$\boxplus$}}}}
      {\vcenter{\hbox{\resizebox{\widthof{$\scriptscriptstyle\oplus$}}{!}{$\boxplus$}}}}
  }\displaylimits
}
\theoremstyle{plain}
\newtheorem{theorem}{Theorem}[section]
\newtheorem{corollary}[theorem]{Corollary}
\newtheorem{lemma}[theorem]{Lemma}
\newtheorem{proposition}[theorem]{Proposition}
\numberwithin{equation}{section}
\theoremstyle{remark}
\newtheorem*{remark}{Remark}
\begin{document}
\title{New integral representations for Rankin--Selberg $L$-functions}
\author{Andrew R.\ Booker}
\email{\tt andrew.booker@bristol.ac.uk}
\thanks{A.~R.~B.\ was partially supported by EPSRC Grant \texttt{EP/K034383/1}.
M.~L.\ was supported by a Royal Society University Research Fellowship.
No data were created in the course of this study.}
\author{M.\ Krishnamurthy}
\email{muthu-krishnamurthy@uiowa.edu}
\author{Min Lee}
\email{\tt min.lee@bristol.ac.uk}
\address{School of Mathematics\\University of Bristol\\
University Walk\\Bristol\\BS8 1TW\\United Kingdom}
\address{Department of Mathematics\\University of Iowa\\14 MacLean
Hall\\Iowa City, IA 52242-1419\\USA}

\begin{abstract}
We derive integral representations for the Rankin--Selberg $L$-functions
on $\GL(3)\times\GL(1)$ and $\GL(3)\times\GL(2)$ by a process of unipotent
averaging at archimedean places. A key feature of our result is that it
allows one to fix the choice of test vector at finite places, irrespective
of ramification. This enables a new proof of the functional equation
for $\GL(3)\times\GL(2)$ Rankin--Selberg $L$-functions in many cases.
\end{abstract}
\maketitle
\section{introduction}\label{intro}
Let $f(z)=\sum_{n=1}^\infty a_n e^{2\pi i n z}$
be a holomorphic newform of weight $k$ and level $\Gamma_0(N)$.
Let $\chi$ be a primitive Dirichlet character of conductor $q$ and set
$f_\chi(z)=\sum_{n=1}^\infty a_n\chi(n)e^{2\pi in z}$.
Then, by \cite[Lemma 4.3.10]{Mi}, we have
$$
f_\chi=\frac{\tau(\chi)}{q}\sum_{a=1}^q\overline{\chi(-a)}
f\Bigr|_k\begin{psmallmatrix}1&a/q\\&1\end{psmallmatrix},
$$
where $\tau(\chi)$ is the associated Gauss sum.
Taking the Mellin transform of both sides, we obtain
\begin{equation}\label{eq:addtomult}
\Lambda(s,f\times\chi)=\frac{\tau(\chi)}{q}
\sum_{a=1}^q\overline{\chi(-a)}\Lambda(s,f,a/q),
\end{equation}
where $\Gamma_\C(s)=2(2\pi)^{-s}\Gamma(s)$,
$$
\Lambda(s,f,a/q)=\Gamma_\C(s)\sum_{n=1}^\infty a_ne^{2\pi ian/q}n^{-s}
\quad\text{and}\quad
\Lambda(s,f\times\chi)=\Gamma_\C(s)\sum_{n=1}^\infty a_n\chi(n)n^{-s}.
$$

The above notion of \emph{unipotent averaging} is a key ingredient in
Weil's proof of the \emph{converse theorem} \cite[Theorem 4.3.15]{Mi}
as well as in our
prior work \cite{B-Kr, B-Kr2} on $\GL(2)$ converse theorems. The goal of
the present paper is to prove an analogue of \eqref{eq:addtomult} for
higher-rank
Rankin--Selberg $L$-functions using a similar process of unipotent
averaging; see Theorem~\ref{main1}.
Our result can be interpreted as
an integral representation for the Rankin--Selberg $L$-function
in terms of the associated (Whittaker) \emph{essential functions}, and
we speculate that this may have applications to the study of
algebraicity properties and nonvanishing of special values.
We outline some necessary notation and explain our result in
greater detail below.

Let $F$ be a number field with ad\`ele ring $\A_F$, and let $\pi\cong
\bigotimes\pi_v$ (resp.\ $\tau\cong\bigotimes\tau_v$) be an irreducible generic
(with respect to some fixed additive character $\psi=\bigotimes\psi_v$
of $F\backslash\A_F$) automorphic representation of $\GL_n(\A_F)$
(resp.\ $\GL_m(\A_F)$), with $m<n$. For each place $v$ of $F$, let
$L(s,\pi_v\boxtimes\tau_v)$ be the local Rankin--Selberg factor
\cite{JPSS3, J2} and set
$$
\Lambda(s,\pi\boxtimes\tau)=\prod_vL(s,\pi_v\boxtimes\tau_v).
$$
The completed $L$-function
$\Lambda(s,\pi\boxtimes\tau)$, definied initially for $\Re(s)\gg1$,
continues to a meromorphic function of $s\in\C$ and satisfies
a functional equation relating $\Lambda(s,\pi\boxtimes\tau)$ to
$\Lambda(1-s,\widetilde{\pi}\boxtimes\widetilde{\tau})$
\cite{JPSS3}. The collection of such functions for a fixed
$\pi$ and suitably varying $\tau$ plays a central role in {converse
theorems}, whose purpose is to characterize automorphic representations
$\pi$ in terms of the analytic behavior of $\Lambda(s,\pi\boxtimes\tau)$.

For each non-archimedean place
$v$ of $F$, let $\xi_v^0$ (resp.\ $\varphi_v^0$) denote the ``essential
vector'' in the space of $\pi_v$ (resp.\ $\tau_v$) \cite{JPSS,
J1, Mat}, and let $W_{\xi_v^0}\in\W(\pi_v,\psi_v)$
(resp.\ $W_{\varphi_v^0}\in\W(\tau_v,\psi_v^{-1})$) be
the associated essential Whittaker functions (as in loc.~cit.) formed
relative to a $\psi_v$ whose conductor is $\o_v$. When $\tau_v$
is unramified, it follows from \cite[Corollary 3.3]{Mat} that
\begin{equation}\label{unr3}
L(s,\pi_v\boxtimes\tau_v)=\int_{\U_m(F_v)\backslash\GL_m(F_v)}
W_{\xi_v^0}\!\begin{pmatrix}h_v&\\&I_{n-m}\end{pmatrix}
W_{\varphi_v^0}(h_v)\|\det{h_v}\|_v^{s-\frac{n-m}{2}}\,dh_v
\end{equation}
for a suitable normalization of the measure on
$\U_m(F_v)\backslash\GL_m(F_v)$. For $m=n-1$, the above equality is a
part of the characterization of the essential vector in \cite{JPSS,J1},
which we review in the next section; the fact that it is true for any
$m\leq n-1$ is a result of a concrete realization of essential functions
in \cite{Mat}. It follows from this that, if $\tau$ is such that $\tau_v$
is unramified at every finite $v$, then the Rankin--Selberg $L$-function
$\Lambda(s,\pi\boxtimes\tau)$ can be represented by a global integral
of the corresponding automorphic forms, which unfolds to a product of
local integrals of the above shape.

However, when $\tau$ is ramified at a place $v$, the local integral
can vanish unless $\xi_v^0$ is adjusted in a suitable way. In fact
in general one has that there is an element
$t_v^0\in\W(\pi_v,\psi_v)\otimes\W(\tau_v,\psi_v^{-1})$
whose local zeta integral is precisely the $L$-function
$L(s,\pi_v\boxtimes\tau_v)$. Now, let $S$ be the finite set
of finite places where $\tau_v$ is ramified. For $v\in S$, let
$t_v^0$ be as described above, and for a finite place $v$ not in $S$ set
$t_v^0=W_{\xi_v^0}\otimes W_{\varphi_v^0}$. For any choice of a pair
of archimedean Whittaker functions $(W_\infty, W_\infty')$, writing each
$t_v^0$ for $v\in S$ as a sum of pure tensors, we get a finite sum of pure
(global) tensors such that
\[
\sum_j W_j\otimes W_j'=(W_\infty\otimes W_\infty')\prod_{v<\infty}t_v^0.
\]
Fixing such a decomposition, let $(\phi_j,\varphi_j)$ be the automorphic
forms associated with $(W_j,W_j')$. Then
\begin{equation}\label{eqn-ir}
P_\infty(s)\Lambda(s,\pi\boxtimes\tau)
=\sum_j\int_{\GL_m(F)\backslash\GL_m(\A_F)}
\PP_m^n(\phi_j)\!\left(\begin{pmatrix}h&\\&I_{n-m}\end{pmatrix}\right)
\varphi_j(h)\|\det h\|^{s-\frac{n-m}{2}}\,dh,
\end{equation}
where $P_\infty(s)$ is an entire function depending only on the pair
$(W_\infty,W_\infty')$, and $\PP_m^n$ is the projection operator defined
in \eqref{Pdef} below.

In this paper, we study integral representations for
$L(s,\pi\boxtimes\tau)$ in the case $n=3$, $m\in\{1,2\}$.  We show
that (in the above notation) one can take $t_v^0=W_{\xi_v^0}\otimes
W_{\varphi_v^0}$ at all finite places $v$, regardless of the ramification
of $\tau$, after going through a process of unipotent averaging
on the right-hand side of \eqref{eqn-ir}. To illustrate our main
theorem (Theorem~\ref{main1}), let us consider the classical case, $F=\Q$.
Let $N,q\geq 1$ be the conductors of $\pi$ and $\tau$, respectively,
and assume for simplicity that $(N,q)=1$ and $\tau_v$ is a twist-minimal
principal series representation for every finite place $v$.
We fix the essential
vector $\xi_f^0=\bigotimes_{v<\infty}\xi_v^0$ and consider pure
tensors $\xi$ in the space of $\pi$ of the form $\xi=\xi_\infty\otimes
\xi_f^0$. Let $U_\xi$ be the function on $\GL_3(\A_\Q)$
obtained by summing the associated Whittaker function $W_\xi$ as
in \eqref{U}. Then $U_\xi$ may be
viewed as an automorphic function on $\GL_2(\A_\Q)$ by restriction via
$h\mapsto\begin{psmallmatrix}h&\\&1\end{psmallmatrix}$. It in fact
corresponds to a \emph{classical} cusp form of level $1$. If we write
$\Phi$ to denote the restriction of $U_\xi$ to
$\GL_3(\R)$, then the corresponding classical automorphic
form on $\GL_2(\R)$ is given by $h\mapsto
\Phi\begin{psmallmatrix}h&\\&1\end{psmallmatrix}$
for $h\in\GL_2(\R)$. Our key
observation is that by shifting this
function by a rational unipotent element, one can produce automorphic
forms (on $\GL_2(\R)$) of arbitrary level. To be precise, let
$(\beta_1,\beta_2)\in\Q^2$ with common denominator $q$, and consider
the function
$$
h\mapsto
\Phi\!\begin{pmatrix}h&{}^t\beta\\&1\end{pmatrix},
\quad\text{for }h\in\GL_2(\R).
$$
As we show in Section~\ref{fourier}, this is an automorphic form with
respect to the principal congruence subgroup $\Gamma(q)$ of level $q$.
If $\varphi$ is an automorphic function in the space of $\tau$ that
corresponds to a pure tensor whose component at every finite place $v$
is $\varphi_v^0$, put
\[
\Lambda(s,\pi,\tau,\xi_\infty,\varphi_\infty,\beta)
=\int_{\Gamma(q)\backslash\GL_2(\R)}
\Phi\!\begin{pmatrix}h&{}^t\beta\\&1\end{pmatrix}
\varphi(h)\|\det\,h\|_{\R}^{s-\frac12}\,dh.
\]

We can paraphrase Theorem~\ref{main1} in this context as saying
that
\[
P_\infty(s)\Lambda(s,\pi\boxtimes\tau)=
\Lambda(s,\pi,\tau,\xi_\infty,\varphi_\infty,(0,1/q)),
\]
for a certain entire function $P_\infty(s)$, depending on the pair
$(\xi_\infty,\varphi_\infty)$. (Notice the similarity with \eqref{eqn-ir},
albeit here the integral on the right-hand side involves classical
forms.) Further, applying Theorem~\ref{main1} we deduce the analytic
continuation and functional equation of $\Lambda(s,\pi\boxtimes\tau)$
(cf.\ Theorem~\ref{main2}). Our proof is different from the proof of
Jacquet, Piatetski-Shapiro and Shalika \cite{JPSS3}, in that it does not
use the local functional equations at ramified places. The argument
is analogous to the proof of the functional equation for additive twists
on $\GL(2)$ (see \cite[\S{A.3--A.4}]{KMV}), and it is based on
a higher-rank matrix identity appearing in \cite{LM}.

Theorem~\ref{main1} is similar in
spirit to the \emph{global Birch lemma} of \cite{Jan, KMS}, which plays
a central role in the study of $\p$-adic interpolation of special values
of $L(s,(\pi\otimes\chi)\boxtimes\tau)$ for a cuspidal pair $(\pi,\tau)$
and a Hecke character $\chi$ of finite order and non-trivial $\p$-power
conductor. While in loc.~cit.\ the unipotent averaging is done at the
finite place corresponding to $\p$, here we do it over the archimedean
places as in \cite{B-Kr}, which allows us to retain the essential vector
at all finite places. In addition, it should be emphasized that $\chi$
can be taken to be trivial in our setting.
\subsection*{Acknowledgements}
The second author (M.~K.) would like to thank A.~Raghuram for pointing him to \cite{Jan}.

\section{preliminaries}\label{prelim}
Let $\o_F$ denote the ring of integers of $F$. For
each place $v$ of $F$, let $F_v$ denote the completion of $F$ at
$v$. For $v<\infty$, let $\o_v$ denote the ring of integers in $F_v$,
$\p_v$ the unique maximal ideal in $\o_v$ and $q_v$ the cardinality
of $\o_v/\p_v$. Also, for every finite $v$, fix a generator
$\varpi_v$ of $\p_v$ with absolute value $\|\varpi_v\|_v=q_v^{-1}$. Let
$F_\infty=\prod_{v\mid\infty}F_v$ and let $\A_F=F_\infty\times\A_{F,f}$
denote the ring of ad\`eles of $F$. Throughout the paper, we fix
an {\it unramified} additive character $\psi=\bigotimes\psi_v$ of $F\backslash\A_F$ whose local conductor at every finite place $v$ is $\o_v$, i.e., $\psi_v$ is trivial on $\o_v$ but non-trivial on $\p_v^{-1}$.  We write $\psi_{\infty}$ to denote the character $\prod_{v|\infty}\psi_v$ of $F_{\infty}$.

For any $n>1$, let $\B_n=\T_n\text{U}_n$ be the Borel subgroup of $\GL_n$
consisting of upper triangular matrices; let $\P_n'\supset\B_n$ be the
standard parabolic subgroup of type $(n-1,1)$ with Levi decomposition
$\P_n'=\text{M}_n\text{N}_n$. Then $\text{M}_n\cong\GL_{n-1}\times\GL_1$ and
\[
\text{N}_n=\left\{\begin{pmatrix}I_{n-1}&\ast\\&1\end{pmatrix}\right\}.
\]
If $R$ is any $F$-algebra and $H$ is any algebraic $F$-group, we write
$H(R)$ to denote the corresponding group of $R$-points. Let $\P_n(R)\subset
\P_n'(R)$ denote the mirabolic subgroup consisting of matrices whose last
row is of the form $(0,\ldots,0,1)$, i.e.,
$$
\P_n(R)=\left\{\begin{pmatrix}h&y\\&1\end{pmatrix}:
h\in\GL_{n-1}(R), y\in R^{n-1}\right\}
\cong\GL_{n-1}(R)\ltimes \text{N}_n(R).
$$
Let $w_n$ denote the long Weyl element in $\GL_n(R)$. For $1\leq m\leq n-1$, let $\alpha_m$ denote the permutation matrix
\[
\alpha_m=\begin{pmatrix}&1&\\I_m&&\\&&I_{n-m-1}\end{pmatrix}.
\]
Note that $\alpha_{n-1}=w_n\begin{psmallmatrix}w_{n-1}&\\&1\end{psmallmatrix}$.
The character
\[
u\mapsto\psi\!\left(\sum u_{i,i+1}\right)\quad\text{for }u\in \text{U}_n(\A)
\]
defines a \emph{generic character} of $\U_n(F)\backslash\U_n(\A_F)$,
and by abuse of notation we continue to denote this character by
$\psi$. Similarly, for each $v$, we also obtain a generic character
$\psi_v$ of $\U_n(F_v)$. Further, for any algebraic subgroup
$V\subseteq\U_n$, $\psi$ and $\psi_v$ define characters of $V(\A_F)$ and
$V(F_v)$, respectively, via restriction. In particular, we may consider
the character $\psi|_{\text{N}_n(\A_F)}$; its stabilizer in $\text{M}_n(\A_F)$ is
then $\P_{n-1}(\A_F)$, where we regard $\P_{n-1}$ as a subgroup of $\text{M}_n$
via $h\mapsto\begin{psmallmatrix}h&\\&1\end{psmallmatrix}$.

For each $v<\infty$, we will consider certain compact open subgroups
of $\GL_n(F_v)$; namely, let $K_v=\GL_n(\o_v)$, and for any integer
$m\ge 0$, set
\begin{align*}
K_{1,v}(\p_v^m)&=\left\{g\in\GL_n(\o_v):g\equiv
\begin{psmallmatrix}&&&\ast\\&\ast&&\vdots\\&&&\ast\\0&\cdots&0&1\end{psmallmatrix}
\pmod*{\p_v^m}\right\},\\
K_{0,v}(\p_v^m)&=\left\{g\in\GL_n(\o_v):g\equiv
\begin{psmallmatrix}&&&\ast\\&\ast&&\vdots\\&&&\ast\\0&\cdots&0&\ast\end{psmallmatrix}
\pmod*{\p_v^m}\right\},
\end{align*}
so that $K_{1,v}(\p_v^m)$ is a normal subgroup of $K_{0,v}(\p_v^m)$,
with quotient $K_{0,v}(\p_v^m)/K_{1,v}(\p_v^m)\cong(\o_v/\p_v^m)^\times$. Let $K_f=\prod_{v<\infty}K_v$, and for an integral ideal $\a$
of $F$, set
$$
K_i(\a)=\prod_{v<\infty}K_{i,v}(\p_v^{m_v})\quad\text{for }i=0,1,
$$
where $m_v$ are the unique non-negative integers such that
$\a=\prod_v(\p_v\cap\o_F)^{m_v}$.
Then $K_1(\a)\subseteq K_0(\a)\subseteq K_f$ are compact open subgroups of
$\GL_n(\A_{F,f})$. We may also form the corresponding principal
congruence subgroups of $\GL_n(F_\infty)$, embedded diagonally, namely,
$$
\Gamma_i(\a)=\{\gamma\in\GL_n(F):\gamma_f\in K_i(\a)\}
\subset\GL_n(F_\infty)\quad
\text{for }i=0,1,
$$
where $\gamma_f$ denotes the image of $\gamma$ in $\GL_n(\A_{F,f})$.

From strong approximation for $\GL_n$, it follows that the set
$\GL_n(F)\backslash\GL_n(\A_F)/\GL_n(F_\infty)K_1(\a)$ of double cosets
is finite of cardinality $h$, where $h$ is the class number of $F$.
Let us write
\begin{equation}\label{sa}
\GL_n(\A_F)=\coprod_{j=1}^h\GL_n(F)g_j\GL_n(F_\infty)K_1(\a),
\end{equation}
where each $g_j\in\GL_n(\A_{F,f})$. Then
\begin{equation}\label{sa2}
\GL_n(F)\backslash\GL_n(\A_F)/K_1(\a)\cong
\coprod_{j=1}^h\Gamma_{1,j}(\a)\backslash\GL_n(F_\infty),
\end{equation}
where $\Gamma_{1,j}(\a)=\{\gamma\in\GL_n(F):
\gamma_f\in g_j\,K_1(\a)g_j^{-1}\}
\subset\GL_n(F_\infty)$, embedded diagonally.
Replacing $K_1(\a)$ by $K_0(\a)$ in this definition,
we get the corresponding groups $\Gamma_{0,j}(\a)$.

Before proceeding further, let us introduce our choice of
measures. For each place $v$ of $F$ and $n\geq 1$, we normalize the
Haar measure on $\GL_n(F_v)$ and $K_v$ so that $\vol(K_v)=1$. We
fix the Haar measure on $\U_n(F_v)$ for which $\vol(\U_n(F_v)\cap
K_v)=1$. From these measures, we obtain a right-invariant measure
on $\U_n(F_v)\backslash\GL_n(F_v)$. Globally, on $\GL_n(\A_F)$
and $\U_n(\A_F)$, respectively, we take the corresponding product
measure. On the compact quotient $\U(F)\backslash\U_n(\A_F)$, we choose
the right-invariant measure of unit volume. Since
$\U_2(\A_F)\cong\A_F$ and $\A_F=F+F_\infty+\prod_{v<\infty}\o_v$,
it follows from our normalization that $\text{vol}(\o_F\backslash
F_\infty)=1$.

Suppose $(\pi, V_\pi)$ is an irreducible admissible representation
of $\GL_n(\A_F)$ (not necessarily automorphic) whose central
character $\omega_\pi$ is an id\`ele class character. We write
$\pi\cong\bigotimes\pi_v$ as a restricted tensor product with respect
to a distinguished set of spherical vectors, where for each finite $v$,
$\pi_v$ is an irreducible admissible representation of $\GL_n(F_v)$
on a complex vector space $V_{\pi_v}$, and for each $v\mid\infty$,
$\pi_v$ is an irreducible admissible Harish-Chandra module. However,
for $v\mid\infty$, we can pass to the canonical completion of $\pi_v$
in the sense of Casselman and Wallach (see \cite{J2}) to obtain a smooth
representation of moderate growth of $\GL_n(F_v)$. By abuse of notation we
continue to write $(\pi_v,V_{\pi_v})$ to denote the canonical completion
at an archimedean place $v$ and let $(\pi_\infty, V_{\pi_\infty})$
denote the topological tensor product of these representations. Thus
we may (and we will) take $V_\pi$ to be the restricted tensor product
$V_{\pi_\infty}\otimes\bigotimes_{v<\infty}V_{\pi_v}$ and $\pi$ the
corresponding representation of $\GL_n(\A_F)$.

Now with $\pi\cong\bigotimes\pi_v$ as in the previous paragraph,
let us further assume that each $\pi_v$ is generic, meaning there is
a non-zero linear form (continuous when $v\mid\infty$) $\lambda_v:
V_{\pi_v}\to\C$ satisfying $\lambda_v(\pi_v(n)w)=\psi_v(n)\lambda_v(w)$
for $w\in V_{\pi_v}, n\in\text{N}_n(F_v)$. Let $\W(\pi_v,\psi_v)$ denote the
Whittaker model of $\pi_v$, viz.\ the space of functions on $\GL_n(F)$
defined by $g\mapsto\lambda_v(\pi_v(g)w)$ for fixed $w\in V_{\pi_v}$. By
taking the tensor product of $\lambda_{\pi_v}$ for $v\mid\infty$,
we can also form the space $\W(\pi_\infty,\psi_\infty)$. We note that
for $v<\infty$ where $\pi_v$ is unramified, there is a unique vector
$W_{\pi_v}^0\in\W(\pi_v,\psi_v)$ that is fixed under $K_v$ and takes the
value $1$ at the identity. The global Whittaker model $\W(\pi,\psi)$ of
$\pi$ is the space spanned by the functions $W_\infty\prod_{v<\infty}W_v$
with $W_\infty\in\W(\pi_\infty,\psi_\infty)$, $W_v\in\W(\pi_v,\psi_v)$
and $W_v=W_{\pi_v}^0$ for almost all $v$. For every $\eta\in V_\pi$
we denote by $W_\eta$ the corresponding element of $\W(\pi,\psi)$.

Following \cite{CPS1, JPSS2}, for each $\eta\in V_\pi$ we set
\begin{equation}\label{U}
U_\eta(g)=\sum_{\gamma\in\U_n(F)\backslash\P_n(F)}W_\eta(\gamma g),
\end{equation}
which can also be written as
\[
U_\eta(g)=\sum_{\gamma\in\U_{n-1}(F)\backslash\GL_{n-1}(F)}
W_\eta\!\left(\begin{pmatrix}\gamma&\\&1\end{pmatrix}g\right).
\]
It is shown in \cite[Section 12]{JPSS2} that this sum converges
absolutely and uniformly on compact subsets to a continuous function
on $\GL_n(\mathbb A_F)$, and that it is cuspidal along the unipotent
radical of any standard maximal parabolic subgroup of $\GL_n$.
Further, $U_\eta$ is left invariant under both $\P_n(F)$ and the center
$Z_n(F)$. For $W\in\W(\pi,\psi)$, let
$\widetilde{W}(g)=W(w_n{}^tg^{-1})\in\W(\tilde{\pi},\psi^{-1})$.
For each place $v$, if $W_v\in\W(\pi_v,\psi_v)$, we similarly define
$\widetilde{W}_v\in\W(\pi_v,\psi_v^{-1})$. Let $V_\eta$ (cf.~\cite{CPS1})
denote the function
\[
V_\eta(g)=\sum_{\gamma\in\U_n'(F)\backslash\operatorname{Q}_n(F)}
W_\eta(\alpha_{n-1}\gamma g),
\]
where 
$\operatorname{Q}_n=\overline{\P}_n$ is the opposite mirabolic subgroup
and $\U_n'=\alpha_{n-1}^{-1}\U_n\alpha_{n-1}$. If we put
\[
\widetilde{U}_\eta(g)=\sum_{\gamma\in\U_{n-1}(F)\backslash\GL_{n-1}(F)}
\widetilde{W}_\eta\!\left(\begin{pmatrix}\gamma&\\&1\end{pmatrix}g\right).
\]
Then as explained in \cite{B-Kr4}, one has the equality
$\widetilde{U}_\eta(^t{}g^{-1})=V_\eta(g)$.
The following is a basic result of Cogdell
and Piatetski-Shapiro \cite{CPS1,CPS2}:
\begin{proposition}\label{prop-CPS}
Suppose $\pi$ is as above with automorphic central character $\omega_\pi$.
Then $\pi$ is cuspidal automorphic if and only if $U_\eta=V_\eta$ for all
$\eta\in V_\pi$.
\end{proposition}
For $m<n$, let $Y=Y_{n,m}$ be the unipotent radical of the standard
parabolic subgroup of $\GL_n$ of type $(m+1,1,\ldots,1)$. (In this
notation, $\text{N}_n=Y_{n,n-2}$.) For a function $f$ on $\GL_n(\A_F)$ that is
left invariant under $Y(F)$, we set
\begin{equation}\label{Pdef}
\PP_m^n(f)(g)=\int_{Y(F)\backslash Y(\A_F)}f(yg)\overline{\psi(y)}\,dy,
\end{equation}
where $dy$ is the Haar measure on $Y(\A_F)$ for which the quotient
$Y(F)\backslash Y(\A_F)$ has unit volume. In particular, we can consider
$\PP_m^n(U_\eta)(g)$, with the following Fourier series expansion:
\begin{equation}\label{fs}
\begin{aligned}
\PP_m^n(U_\eta)(g)&=\sum_{\gamma\in\U_{m+1}(F)\backslash\P_{m+1}(F)}
W_\eta\!\left(\begin{pmatrix}\gamma&\\&I_{n-m-1}\end{pmatrix}g\right)\\
&=\sum_{\gamma\in\U_m(F)\backslash\GL_m(F)}
W_\eta\!\left(\begin{pmatrix}\gamma&\\&I_{n-m}\end{pmatrix}g\right).
\end{aligned}
\end{equation}
Viewed as a function on $\P_{m+1}(\A_F)$, $\PP_m^n(U_\eta)(g)$ is
left invariant under $\P_{m+1}(F)$ and is cuspidal in the sense that all
the relevant unipotent integrals are zero. Observe that $\PP_{n-1}^n$
is the identity.

Suppose $\tau$ is an automorphic subrepresentation of $\GL_m(\A_F)$,
meaning $V_\tau$ is a subspace of the space of automorphic forms on
$\GL_m(\A_F)$. Let $\phi$ be an automorphic form in $V_\tau$ and set
\begin{equation}\label{rs}
I(s;U_\eta,\phi)=\int_{\GL_m(F)\backslash\GL_m(\A_F)}
\PP_m^n(U_\eta)\!\left(\begin{pmatrix}h&\\&I_{n-m}\end{pmatrix}\right)
\phi(h)\|\det h\|^{s-\frac{n-m}{2}}\,dh.
\end{equation}
This integral is absolutely convergent for all $s$, as we are integrating
a cusp form against an automorphic form. For $\Re(s)\gg1$ we can replace
$\PP_m^n(U_\eta)$ by its Fourier expansion and unfold, to get
\[
I(s;U_\eta,\phi)=\int_{\U_m(\A_F)\backslash\GL_m(\A_F)}
W_\eta\!\begin{pmatrix}h&\\&I_{n-m}\end{pmatrix}
W_\phi(h)\|\det h\|^{s-\frac{n-m}{2}}\,dh,
\]
where $W_\phi$ is the Whittaker on $\GL_m(\A_F)$
associated to $\phi$ with respect to $\psi^{-1}$, i.e.\
\[
W_\phi(h)=\int_{\U_m(F)\backslash\U_m(\A_F)}\phi(uh)\psi(u)\,du.
\]

Now, if $\eta$ (resp.\ $\phi$) corresponds to a pure tensor under
$\pi\cong\bigotimes\pi_v$ (resp.\ $\tau\cong\bigotimes\tau_v$), then by
the uniqueness of the Whittaker model we have
\[
W_\eta(g)=\prod_vW_{\eta_v}(g),\quad W_\phi(g)=\prod_vW_{\phi_v}(g_v),
\]
and the integral now factors as
$$
I(s;U_\eta,\phi)=\prod_v\Psi_v(s;W_{\eta_v},W_{\phi_v}),
$$
where the local integrals are given by
$$
\Psi_v(s;W_{\eta_v},W_{\phi_v})=\int_{\U_m(F_v)\backslash\GL_m(F_v)}
W_{\eta_v}\!\begin{pmatrix}h_v&\\&I_{n-m}\end{pmatrix}
W_{\phi_v}(h_v)\,\|\det{h_v}\|_v^{s-\frac{n-m}{2}}\,dh_v.
$$
It follows from the Rankin--Selberg theory of local factors that
\[
E_v(s):=\frac{\Psi_v(s;W_{\eta_v},W_{\phi_v})}{L(s,\pi_v\boxtimes\tau_v)}
\]
is entire for all $v$ \cite{JPSS3, J2}. If $v$ is non-archimedean then
$E_v(s)\in\C[q_v^s,q_v^{-s}]$, and $E_v(s)=1$ for almost all finite $v$.
Setting $E(s)=\prod_vE_v(s)$, for pure tensors $\eta$ and $\phi$ as
above one has
\[
I(s;U_\eta,\phi)=E(s)\prod_vL(s,\pi_v\boxtimes\tau_v).
\]

One also has the analogous integrals involving the function $V_\eta$.
Namely, put $\alpha_m'=\alpha_{n-1}^{-1}\alpha_m$, $1\leq m\leq n-1$,
and consider the function (cf.~\cite{Cog})
\[
V_\eta^m(g)=V_\eta(\alpha_m'g)\quad\text{for }g\in\GL_n(\A_F).
\]
This function is clearly left invariant under
$\alpha_m'^{-1}Q_n(F)\alpha_m'$ which contains the group $Y_{n,m}(F)$.
Hence $\PP_m^n(V_\eta^m)$ is well defined, and we put
\begin{equation}\label{rs1}
I(s;V_\eta,\phi)=\int_{\GL_m(F)\backslash\GL_m(\A_F)}
\PP_m^n(V_\eta^m)\!\left(\begin{pmatrix}h&\\&I_{n-m}\end{pmatrix}\right)
\phi(h)\|\det h\|^{s-\frac{n-m}{2}}\,dh.
\end{equation}
As explained in loc.~cit., this integral converges for $-\Re(s)\gg1$ and
unfolds to give
\[
I(s;V_\eta,\phi)=\prod_v\widetilde{\Psi}_v(1-s;\rho(w_{n,m})
\widetilde{W}_{\eta_v},\widetilde{W}_{\phi_v}),
\]
where $\rho$ denotes right translation,
$w_{n,m}=\begin{psmallmatrix}1&\\&w_{n-m}\end{psmallmatrix}$, and
\[
\widetilde{\Psi}_v(s;W,W')=\int_{\U_m(F_v)\backslash\GL_m(F_v)}
\int_{M_{n-m-1,m}(F_v)}W\!\begin{pmatrix}h&&\\x&I_{n-m-1}&\\&&1\end{pmatrix}
W'(h)\|\det h\|_v^{s-\frac{n-m}{2}}\,dx\,dh.
\]
Further, the local $\epsilon$-factor $\epsilon(s,\pi_v\boxtimes\tau_v,\psi_v)$
is defined so that the following local functional equation
\cite[Theorem 3.2]{Cog} holds:
\begin{equation}\label{lfe}
\frac{\widetilde{\Psi}(1-s;\rho(w_{n,m})\widetilde{W}_{\eta_v},\widetilde{W}_{\phi_v})}
{L(1-s,\widetilde{\pi}_v\boxtimes\widetilde{\tau}_v)}
=\omega_{\tau_v}(-1)^{n-1}\epsilon(s,\pi_v\boxtimes\tau_v,\psi_v)
\frac{\Psi_v(s;W_{\eta_v},W_{\phi_v})}{L(s,\pi_v\boxtimes\tau_v)}.
\end{equation}
If $\tau=\mathbf{1}$ is the trivial representation of $\GL_1(\A_F)$, then
for each $v$, we write $L(s,\pi_v)$ and $\epsilon(s,\pi_v,\psi_v)$
to denote $L(s,\pi_v\times\mathbf{1}_v)$ and
$\epsilon(s,\pi_v\times\mathbf{1}_v,\psi_v)$, respectively. These are known
to be the same as the standard local factors of Godement--Jacquet attached
to the representation $\pi_v$ and the character $\psi_v$.

With $\pi\cong\bigotimes\pi_v$ as above, let us now recall the notion of
the \emph{essential vector} of $\pi_v$. For each finite $v$, according
to \cite{JPSS} (see also \cite{J1}), there is a unique positive integer
$m(\pi_v)$ such that the space of $K_1\bigl(\p_v^{m(\pi_v)}\bigr)$-fixed
vectors is of dimension $1$.  Further, as mentioned in the introduction,
by loc.~cit., there is a unique vector $\xi_v^0$ in this space, called the
essential vector, with the associated essential function
$W_{\xi_v^0}=W_{\pi_v}$ in ${\mathcal W}(\pi_v,\psi_v)$
satisfying the condition
$W_{\xi_v^0}\begin{psmallmatrix}gh&\\&1\end{psmallmatrix}
=W_{\xi_v^0}\begin{psmallmatrix}g&\\&1\end{psmallmatrix}$
for all $h\in\GL_{n-1}(\o_v)$ and $g\in\GL_{n-1}(F_v)$. For $v<\infty$
and any unramified representation $\tau_v$ of $\GL_m(F_v)$, let
$W_{\tau_v}^{0,\psi_v}\in W(\tau_v,\psi_v)$ denote the normalized
spherical function (cf.~\cite[p.~2]{J1}).  If $m(\pi_v)=0$
then by uniqueness of essential functions one has the equality
$W_{\xi_v^0}=W_{\pi_v}^{0,\psi_v}$ in $W(\pi_v,\psi_v)$. The integral
ideal $\n=\prod_{v<\infty}(\p_v\cap\o_F)^{m(\pi_v)}\subseteq\o_F$
is called the \emph{conductor} of $\pi$. We write $\nid$ to denote a
finite id\`ele corresponding to $\n$. In fact, we may fix $\nid$ so that
$\nid_v=\varpi_v^{m(\pi_v)}$ for all finite $v$. The integer $m(\pi_v)$
can also be characterized as the degree of the monomial in the local
$\epsilon$-factor $\epsilon(s,\pi_v,\psi_v)$ \cite{JPSS}; we write
\begin{equation}\label{rn}
\epsilon(s,\pi_v,\psi_v)=\epsilon(\pi_v,\psi_v)q_v^{m(\pi_v)(\frac12-s)}.
\end{equation}

A crucial property of the conductor is that $K_0(\n)$ acts on
the space of $K_1(\n)$-fixed vectors via the central character
$\omega=\omega_\pi$ (cf.~\cite[Section 8]{CPS1}). Precisely, for
$g_v=(g_{i,j})\in K_{0,v}\bigl(\p_v^{m(\pi_v)}\bigr)$, define
\[
\chi_{\pi_v}(g_v)=\begin{cases}
1&\text{if }m(\pi_v)=0,\\
\omega_v(g_{n,n})&\text{if }m(\pi_v)>0,
\end{cases}
\]
and put $\chi_\pi=\bigotimes_{v<\infty}\chi_{\pi_v}$. Let
$\xi_f^0=\bigotimes_{v<\infty}\xi_v^0$, which forms a basis for the space
of $K_1(\n)$-fixed vectors in $V_\pi$. Then it is shown in loc.~cit.\
that $\chi_\pi$ is a character of $K_0(\n)$ trivial on $K_1(\n)$,
and that for any finite $v$,
\[
\pi_v(g)\xi_v^0=\chi_{\pi_v}(g)\xi_v^0
\quad\text{for all }g\in K_{0,v}\bigl(\p_v^{m(\pi_v)}\bigr).
\]
For each $j$, $\chi_\pi$ also determines a character of $\Gamma_{0,j}(\n)$
that is trivial on $\Gamma_{1,j}(\n)$, which we continue to denote
by $\chi_\pi$.

We will require the next two lemmas later in \S\ref{fe}. To
that end, for any $v$ and any $\xi_v$ in the space of $\pi_v$, put
\begin{equation}\label{tildedef}
\tilde{\xi}_v=\epsilon(\pi_v,\psi_v)^{1-n}\pi_v
\begin{psmallmatrix}\nid_vI_{n-1}&\\&1\end{psmallmatrix}\xi_v.
\end{equation}
(Note that for $v\mid\infty$, $\nid_v=1$.) First, we identify the essential
function $W_{\widetilde{\pi}_v}\in\W(\widetilde{\pi}_v,\psi_v^{-1})$ associated
to the contragredient representation $\widetilde{\pi}_v$ in the following lemma.
\begin{lemma}\label{dnf}
For $v<\infty$,
\[
\widetilde{W}_{\tilde{\xi}_v^0}(g)=W_{\widetilde{\pi}_v}(g).
\]
\end{lemma}
\begin{proof}
It is a straightforward calculation to check that the function
$g\mapsto\widetilde{W}_{\tilde{\xi}_v^0}(g)$ is right
$K_{1,v}(\p_v^{m(\pi_v)})$-invariant. Indeed one checks that
\[
\widetilde{W}_{\tilde{\xi}_v^{0}}(gk)
=\chi_{\pi_v}(k)^{-1}\widetilde{W}_{\tilde{\xi}_v^0}(g),
\quad\text{for }k\in K_{0,v}(\p_v^{m(\pi_v)}).
\]
Since the space of such functions in $\W(\widetilde{\pi}_v,\psi_v^{-1})$
is $1$-dimensional, there exists a constant $c$ so that
$\widetilde{W}_{\tilde{\xi}_v^0}(g)=cW_{\widetilde{\pi}_v}(g)$.

It remains to show that $c=1$. For that, we apply the
local functional equation \eqref{lfe} with $\eta_v=\tilde{\xi}_v^0$
and $W_{\phi_v}=W_{\tau_v}^{0,\overline{\psi}_v}$ the normalized
spherical function (with respect to $\psi_v^{-1}$) corresponding to
the unramified (unitary) principal series representation
$\tau_v=\Ind(\chi_1,,\ldots,\chi_{n-1})$ of $\GL_{n-1}(F_v)$. Here
the $\chi_j$ are unramified characters of $F_v^\times$ of the form
$\chi_j(x)=\|x\|_v^{it_j}$ for $t_j\in\R$. With these choices, since
$W_{\tilde{\xi}_v^0}=cW_{\widetilde{\pi}_v}$ and $W_{\widetilde{\pi}_v}$
is the essential function for the representation $\widetilde{\pi}_v$,
the left-hand side of \eqref{lfe} reduces to the following equality:
\[
c=\epsilon(s,\pi_v\boxtimes\tau_v,\psi_v)
\frac{\Psi_v(s;W_{\tilde{\xi}_v^0},W^{0,\overline{\psi}_v}_{\tau_v})}
{L(s,\pi_v\boxtimes\tau_v)}.
\]
Through the change of variable $g\mapsto g(\nid_v^{-1}I_{n-1})$ in the integral
defining $\Psi_v(s;W_{\tilde{\xi}_v^0},W^{0,\overline{\psi}_v}_{\tau_v})$,
one checks that
\[
\Psi_v(s;W_{\tilde{\xi}_v^0},W^{0,\overline{\psi}_v}_{\tau_v})
=\omega_{\tau_v}^{-1}(\nid_v)\|\nid_v^{-1}\|_v^{s-\frac12}\epsilon(\pi_v,\psi_v)^{1-n}
\Psi_v(s;W_{\xi_v^0},W^{0,\overline{\psi}_v}_{\tau_v}).
\]
On the other hand, according to \eqref{unr3},
$\Psi_v(s;W_{\xi_v^0},W^{0,\overline{\psi}_v}_{\tau_v})
=L(s,\pi_v\boxtimes\tau_v)$. We also have
\[
\epsilon(s,\pi_v\boxtimes\tau_v,\psi_v)
=\prod_{j=1}^{n-1}\epsilon(s+it_j,\pi_v,\psi_v)
\]
and $\omega_{\tau_v}(\nid_v)=q_v^{m(\pi_v)\sum_jit_j}$.
Putting these together we obtain $c=1$.
\end{proof}

Let $\xi$ be a pure tensor in the space of $\pi$
such that $\xi_v=\xi_v^0$ for all $v<\infty$.
We also define its dual
counterpart; namely, for $W'=\widetilde{W}_{\tilde{\xi}_\infty}\times
\prod_{v<\infty}W_{\widetilde{\pi}_v}\in\W(\widetilde{\pi},\psi^{-1})$,
let
\[
U_\xi'(g)=\sum_{\gamma\in\U_{n-1}(F)\backslash\GL_{n-1}(F)}
W'\!\left(\begin{pmatrix}\gamma&\\&1\end{pmatrix}g\right).
\]

\begin{lemma}\label{aut-dual}
Suppose $\pi$ is a cuspidal automorphic representation of $\GL_n(\A_F)$
and $\xi$ is a pure tensor in the space of $\pi$ such that
$\xi_v=\xi_v^0$ for all finite $v$. Then
$$
U_\xi\!\left(
{}^tg^{-1}\begin{pmatrix}\nid I_{n-1}\\&&1\end{pmatrix}\right)
=\epsilon_\pi^{n-1}U_\xi'(g),
$$
where $\epsilon_\pi=\prod_v\epsilon(\pi_v,\psi_v)$.
\end{lemma}
\begin{proof}
Keeping the above notation, let
$\xi'=\xi_\infty\otimes\bigotimes_{v<\infty}\tilde{\xi}_v^0$. Then
\begin{align*}
U_\xi\!\left({}^tg^{-1}
\begin{pmatrix}\nid I_{n-1}\\&&1\end{pmatrix}\right)
&=\prod_{v<\infty}\epsilon(\pi_v,\psi_v)^{n-1}
\cdot U_{\xi'}\bigl({}^tg^{-1}\bigr)\\
&=\prod_{v<\infty}\epsilon(\pi_v,\psi_v)^{n-1}
\cdot\widetilde{U}_{\xi'}(g)\\
&=\epsilon_\pi^{n-1}U_\xi'(g).
\end{align*}
Here, the first equality is a consequence of the definition of $U_{\xi'}$,
the second equality follows from Proposition~\ref{prop-CPS}, and the
last equality follows from Lemma~\ref{dnf}.
\end{proof}

We will also need a certain auxiliary function related to the local
Rankin--Selberg $L$-factor $L(s,\pi_v\boxtimes\tau_v)$. In order to
introduce this we pass to a more general setting and take $\pi_v$ (resp.\
$\tau_v$) to be an irreducible admissible representation of $\GL_n(F_v)$
(resp.\ $\GL_m(F_v)$), not necessarily the local component of a global
representation. It is known that the local $L$-function $L(s,\pi_v)$
is of the form $P_{\pi_v}(q_v^{-s})^{-1}$, where $P_{\pi_v}\in\C[X]$ has
degree at most $n$ and satisfies $P_{\pi_v}(0)=1$.  We may then find $n$
complex numbers $\{\alpha_{v,i}\}$ (allowing some of them to be zero)
such that
\[
L(s,\pi_v)=\prod_{i=1}^n(1-\alpha_{v,i}q_v^{-s})^{-1}.
\]
We call the set $\{\alpha_{v,i}\}$ the \emph{Langlands parameters}
of $\pi_v$; if $\pi_v$ is spherical, they are the usual Satake
parameters. Let $\{\beta_{v,j}\}$ be the Langlands parameters of
$\tau_v$, and set
\[
L(s,\pi_v\times\tau_v)=\prod_{i,j}(1-\alpha_{v,i}\beta_{v,j}q_v^{-s})^{-1}.
\]
Of course, $L(s,\pi_v\times\tau_v)=L(s,\pi_v\boxtimes\tau_v)$
if both $\pi_v$ and $\tau_v$ are spherical.

In the following lemma, we explore the connection between $L(s,\pi_v\times\tau_v)$ and
$L(s,\pi_v\boxtimes\tau_v)$ in more detail, but let us first recall
the full classification of irreducible admissible representations
of $\GL_n(F_v)$. Let $\AA_n$ denote the set of equivalence classes
of such representations and put $\AA=\bigcup\AA_n$. The essentially
square-integrable representations of $\GL_n(F_v)$ have been classified
by Bernstein and Zelevinsky; they are given as follows: If $\sigma_v$
is an essentially square-integrable representation of $\GL_n(F_v)$,
then there is divisor $a\mid n$ and a supercuspidal representation
$\eta_v$ of $\GL_a(F_v)$ such that if $b=\frac{n}{a}$ and
$Q$ is the standard (upper) parabolic subgroup of $\GL_n(F_v)$ of type
$(a,\ldots,a)$, then $\sigma_v$ can be realized as the unique quotient of
the (normalized) induced representation
\[
\Ind_Q^{\GL_n(F_v)}(\eta_v,\eta_v\|\cdot\|_v,\ldots,\eta_v\|\cdot\|_v^{b-1});
\]
the integer $a$ and the class of $\eta_v$ are uniquely determined by
$\sigma_v$. In short, $\sigma_v$ is parametrized by $b$ and $\eta_v$, and
we denote this by $\sigma_v=\sigma_b(\eta_v)$; further, $\sigma_v$ is
square-integrable (also called ``discrete series'') if and only if the
representation $\eta_v\|\cdot\|_v^{\frac{b-1}{2}}$ of $\GL_a(F_v)$ is unitary.

Now, let $P$ be an upper parabolic subgroup of $\GL_n(F_v)$
of type $(n_1,\ldots,n_r)$. For each $i=1,\ldots,r$,
let $\tau_{i,v}^0$ be a discrete series representation of
$\GL_{n_i}(F_v)$.  Let $(s_1,\ldots,s_r)$ be a sequence of
real numbers satisfying $s_1\geq\cdots\geq s_r$, and put
$\tau_{i,v}=\tau_{i,v}^0\otimes\|\cdot\|_v^{s_i}$ (an essentially
square-integrable representation), for $i=1,\ldots,r$. Then the induced
representation
\[
\xi_v=\Ind_P^{\GL_n(F_v)}(\tau_{1,v}\otimes\cdots\otimes\tau_{r,v})
\]
is said to be a representation of $\GL_n(F_v)$ of Langlands type. If
$\tau_v\in\AA_n$, then it is well known that it is uniquely representable
as the quotient of an induced representation of Langlands type. We
write $\tau_v=\tau_{1,v}\boxplus\cdots\boxplus\tau_{r,v}$ to
denote this realization of $\tau_v$. Thus one obtains a sum operation on
the set $\AA$ \cite[(9.5)]{JPSS3}. It follows easily from the definition
that $L(s,\pi_v\times\tau_v)$ is bi-additive:
\begin{align*}
L(s,\pi_v\times(\tau_v\boxplus\tau_v'))
&=L(s,\pi_v\times\tau_v)L(s,\pi_v\times\tau_v')\\
L(s,(\pi_v\boxplus\pi_v')\times\tau_v)
&=L(s,\pi_v\times\tau_v)L(s,\pi_v'\times\tau_v)
\end{align*}
for all $\pi_v,\pi_v',\tau_v,\tau_v'\in\AA$. The local factor
$L(s,\pi_v\boxtimes\tau_v)$ is also bi-additive in the above sense,
according to \cite[(9.5) Theorem]{JPSS3}.

\begin{lemma}\label{aux}
Given $v<\infty$, suppose $\pi_v\in\AA_n$ and $\tau_v\in\AA_m$. Then
\[
L(s,\pi_v\times\tau_v)=P(q_v^{-s})L(s,\pi_v\boxtimes\tau_v)
\]
for some polynomial $P\in\C[X]$.
\end{lemma}
\begin{proof}
Since $\pi_v$ is a sum of essentially square-integrable representations
and both $L(s,\pi_v\times\tau_v)$ and $L(s,\pi_v\boxtimes\tau_v)$
are additive with respect to this sum, it suffices to prove
the lemma for essentially square-integrable representations
$\pi_v$. So for the remainder of the proof we assume $\pi_v$
is an essentially square-integrable representation of the form
$\pi_v=\sigma_b(\eta_v)$.  We proceed by induction on $m$. If
$m=1$, then $\tau_v=\chi_v$ is a quasi-character of $F_v^\times$
and $L(s,\pi_v\boxtimes\chi_v)=L(s,\pi_v\otimes\chi_v)$, where
$\pi_v\otimes\chi_v$ is the representation of $\GL_n(F_v)$ defined by
$g\mapsto\pi_v(g)\chi_v(\det g)$. If $\chi_v$ is unramified, then
\[
L(s,\pi_v\otimes\chi_v)=L(s,\pi_v\times\chi_v)
\]
and consequently $P=1$; on the other hand if $\chi_v$ is
ramified then $L(s,\pi_v\times\chi_v)=1$, and the assertion follows
since $L(s,\pi_v\otimes\chi_v)^{-1}$ is a polynomial in $q_v^{-s}$.

We now assume $m>1$. Suppose $\tau_v$ is an essentially square-integrable
representation of $\GL_m(F_v)$, say $\tau_v=\sigma_{b'}(\eta_v')$,
$\eta_v'\in\AA_{a'}$ is supercuspidal, and $a'b'=m$. Then the  standard
$L$-factor $L(s,\tau_v)$ is given by $L(s,\tau_v)=L(s+b'-1,\eta_v')$
(see \cite{JPSS3}). Consequently, $L(s,\tau_v)=1$ unless $a'=1$ and
$\eta_v'=\chi_v$ is an unramified quasi-character of $F_v^\times$.
However, if $L(s,\tau_v)=1$ then $L(s,\pi_v\times\tau_v)=1$, and the assertion
of the lemma follows. Hence we may assume $\tau_v=\sigma_m(\chi_v)$
for an unramified quasi-character $\chi_v$ of $F_v^\times$, in which case
\begin{equation}\label{e1}
\begin{aligned}
L(s,\pi_v\times\tau_v)&=L(s,\pi_v\otimes\chi_v\|\cdot\|_v^{m-1})\\
&=L(s,\sigma_b(\eta_v)\otimes\chi_v\|\cdot\|_v^{m-1})\\
&=L(s+m-1+b-1,\eta_v\otimes\chi_v).
\end{aligned}
\end{equation}
On the other hand, it also follows from \cite[(8.2) Theorem]{JPSS3} that
\begin{equation}\label{e2}
L(s,\pi_v\boxtimes\tau_v)=\begin{cases}
\prod_{j=0}^{m-1}L(s+j+b-1,\eta_v\otimes\chi_v)&\text{if }m\leq n,\\
\prod_{i=0}^{b-1}L(s+m-1+i,\eta_v\otimes\chi_v)&\text{if }m>n.
\end{cases}
\end{equation}
Now, from \eqref{e1} and \eqref{e2}, one sees that the ratio
$\frac{L(s,\pi_v\times\tau_v)}{L(s,\pi_v\boxtimes\tau_v)}$ is a polynomial
in $q_v^{-s}$, thus proving the lemma when $\tau_v$ is an essentially
square-integrable representation.

If $\tau_v$ is not essentially square-integrable, then as discussed
above there is a partition $\sum_{i=1}^k m_i=m$, with each $m_i<m$,
and essentially square-integrable representations $\tau_{i,v}$ of
$\GL_{m_i}(F_v)$, $i=1,\ldots,k$, so that $\tau_v=\bigboxplus_i\tau_{i,v}$.
We may then use additivity and apply the induction hypothesis to get
the desired conclusion.
\end{proof}

\begin{corollary}
If $L(s,\pi_v\boxtimes\tau_v)=1$, then either $L(s,\pi_v)=1$
 or $L(s,\tau_v)=1$.
\end{corollary}
\begin{proof}
If $L(s,\pi_v\boxtimes\tau_v)=1$, then Lemma~\ref{aux} implies that
$L(s,\pi_v\times\tau_v)$ is a polynomial in $q_v^{-s}$ and hence must
be $1$. This in turn implies the conclusion.
\end{proof}

\section{$\GL_3\times\GL_2$}\label{GL2}
We keep the notation of the previous section and take $n=3$, $m=2$.
Thus $\pi$ is an irreducible admissible generic representation of
$\GL_3(\A_F)$ with conductor $\n$ and central character $\omega_\pi$, and
$\tau$ is an automorphic subrepresentation of $\GL_2(\A_F)$ with central
character $\omega_\tau$. We fix nonzero integral ideals $\q'\subseteq\q$ and
consider only those $\tau$ of conductor $\q$.

Choose ideals $\a_j\subseteq\o_F$ representing the ideal classes
of $F$ so that $\prod_j\a_j$ is coprime to $\q'$. Let $\aid_j$ be a
finite id\`ele such that $\a_j=(\aid_j)$, and take $\aid_1=1$.  For each
$j$, put $h_j=\begin{psmallmatrix}\aid_j&\\&1\end{psmallmatrix}$ and
$g_j=\begin{psmallmatrix}h_j&\\&1\end{psmallmatrix}\in\GL_3(\A_{F,f})$.
Then $\{g_j\}$ is a valid set of representatives in \eqref{sa} (with
$n=3$). In what follows, we will also consider strong approximation for
$\GL_2(\A_F)$, i.e.\
$$
\GL_2(\A_F)=\coprod_{j=1}^h\GL_2(F)h_j\GL_2(F_\infty)K,
$$
where the compact open subgroup $K\subset\GL_2(\A_{F,f})$ is
either $K_f=\GL_2(\o_F)$, $K_1(\q)$ or $K_0(\q)$. We then have
the group $G_j=\Gamma_{1,j}(\o_F)=\Gamma_{0,j}(\o_F)$ along with
the congruence subgroups $\Gamma_{0,j}(\q)$, $\Gamma_{1,j}(\q)$ and
$\Gamma_j(q)$. Explicitly,
$$
G_j=\left\{\begin{pmatrix}a&b\\c&d\end{pmatrix}:
a,d\in\o_F, c\in\a_j^{-1}, b\in\a_j, ad-bc\in\o_F^\times\right\},
$$
which is precisely the stabilizer (acting on the right)
of the lattice $\o_F\oplus\a_j\subset F\oplus F$, and
\begin{align*}
\Gamma_{0,j}(\q)&=
\left\{\begin{pmatrix}a&b\\c&d\end{pmatrix}\in G_j:c\in\q\a_j^{-1}\right\},\\
\Gamma_{1,j}(\q)&=
\left\{\begin{pmatrix}a&b\\c&d\end{pmatrix}\in\Gamma_{0,j}(\q):
d-1\in\q\right\},\\
\Gamma_j(\q)&=\left\{\begin{pmatrix}a&b\\c&d\end{pmatrix}\in\Gamma_{1,j}(\q):
a-1\in\q, b\in\q\a_j\right\}.
\end{align*}
Note that
$\Gamma_j(\q)\subseteq\Gamma_{1,j}(\q)\subseteq\Gamma_{0,j}(\q)\subseteq G_j$,
and that $\Gamma_j(\q)$ is normal in $G_j$.

For any $\alpha\in F\hookrightarrow F_\infty$ embedded diagonally, we form
the ad\`ele $(\alpha,0)\in F_\infty\times\A_{F,f}$ which we continue to
denote by $\alpha$. In this way, we view $\beta=(\beta_1,\beta_2)\in F^2$
as an element of $\A_F^2$. For $\xi_\infty\in V_{\pi_\infty}$, note
that
\[
h\mapsto U_{\xi_\infty\otimes\xi_f^0}
\begin{pmatrix}h&\\&1\end{pmatrix},\quad\text{for }h\in\GL_2(\A_F),
\]
is a rapidly decreasing automorphic function on $\GL_2(\A_F)$ (cf.\
\cite[Theorem 1.1]{Cog}) which allowed us to consider integrals of
the type \eqref{rs} while defining $\Lambda(s,\pi\boxtimes\tau)$. We
now follow \cite{B-Kr3} in order to define \emph{additive twists}
(see loc.~cit.) in the current context. To that end, we modify
the above function by inserting a unipotent element. To be precise,
with $\beta=(\beta_1,\beta_2)$ as above, consider the function
\[
h\mapsto U_{\xi_\infty\otimes\xi_f^0}\!
\begin{pmatrix}h&{}^t\beta\\&1\end{pmatrix},\quad\text{for }h\in\GL_2(\A_F);
\]
it is rapidly decreasing (by the gauge estimates in loc.~cit.), but as
one can easily check, it is not in general $\GL_2(F)$-invariant, i.e., not an
automorphic function on $\GL_2(\A_F)$.

\subsection{The function $\Phi_{\xi_\infty,j}$ and its Fourier expansion}\label{fourier}
For each $j=1,\ldots,h$, put
\[
\Phi_{\xi_\infty,j}(g)=U_{\xi_\infty\otimes\xi_f^0}((g,g_j)),
\quad\text{for }g\in\GL_n(F_\infty),
\]
and for $\beta_1\in\a\q^{-1}$ and $\beta_2\in\q^{-1}$, consider the function
\begin{equation}\label{phi-j}
h\mapsto
\Phi_{\xi_\infty,j}\!\begin{pmatrix}h&{}^t\beta\\&1\end{pmatrix},
\quad\text{for }h\in\GL_2(F_\infty).
\end{equation}
Our goal in this section is to derive its Fourier series expansion (see
\eqref{exp3}), which
will also reveal it to be (classical) automorphic form on
$\GL_2(F_\infty)$. For $k\in\{1,\ldots,h\}$, choose distinct prime ideals
$\p_k$ and $\p_k'$ such that $\p_k\sim\q'\p_k'\sim\a_k$ and $\p_k$ is
coprime to $\q'\a_1\cdots\a_h$, and let $\alpha_k$ be a generator of the
principal fractional ideal $\p_k^{-1}\q'\p_k'$. Then for each fixed $j$,
$\alpha_k\a_j+\o_F=\p_k^{-1}$ runs through a set of representatives of the
class group of $F$. Let $\pid_k$ be a finite id\`ele such that $(\pid_k)=\p_k$.

\begin{lemma}
\label{clgp}
Let $R_1$, $R_2$ be sets of representatives for $F^\times/\o_F^\times$.
Then for any fixed $j$,
$$\left\{
\begin{pmatrix}\gamma_1\gamma_2&\\&\gamma_1\end{pmatrix}
\begin{pmatrix}1&\\\alpha_k&1\end{pmatrix}
:\gamma_1\in R_1, \gamma_2\in R_2, 1\le k\le h\right\}$$
is a set of representatives for $\U_2(F)\backslash\GL_2(F)/G_j$.
\end{lemma}
\begin{proof}
For $M=\begin{psmallmatrix}a&b\\c&d\end{psmallmatrix}\in\GL_2(F)$,
let $I_j(M)$ denote the fractional ideal $c\a_j+d\o_F$. It is easy to
see that $I_j(M)$ depends only on the coset of $M$ in
$\U_2(F)\backslash\GL_2(F)/G_j$, i.e.\
$I_j(uMg)=I_j(M)$ for all $u\in\U_2(F)$, $g\in G_j$.
For a given $M$ there is a unique choice of
$k\in\{1,\ldots,h\}$ and $\gamma_1\in R_1$ such that
$I_j(M)=\gamma_1\p_k^{-1}$.
Thus, $I_j(\gamma_1^{-1}M)=\p_k^{-1}=\alpha_k\a_j+\o_F$.

Suppose that $\gamma_1^{-1}M$ has bottom row
$\begin{pmatrix}c&d\end{pmatrix}$.
Then $c\a_j+d\o_F=\alpha_k\a_j+\o_F$,
which in turn implies
$c\o_F+d\a_j^{-1}=\alpha_k\o_F+\a_j^{-1}$.
Thus, we have
\begin{equation}
\label{bottomrow}
\begin{pmatrix}\alpha_k&1\end{pmatrix}
=\begin{pmatrix}c&d\end{pmatrix}\begin{pmatrix}A&B\\C&D\end{pmatrix}
\end{equation}
for some $A\in\o_F$, $B\in\a_j$, $C\in\a_j^{-1}$, $D\in\o_F$.
It follows that the determinant $AD-BC$ is an element of $\o_F$,
but it need not be a unit.  However, the choice of
$\begin{psmallmatrix}A&B\\C&D\end{psmallmatrix}$
in \eqref{bottomrow} is not unique. For any particular solution
$\begin{psmallmatrix}A&B\\C&D\end{psmallmatrix}$,
it is straightforward to see that
$\begin{psmallmatrix}A'&B'\\C'&D'\end{psmallmatrix}$
is another solution if and only if
$\begin{psmallmatrix}A'\\C'\end{psmallmatrix}=
\begin{psmallmatrix}A\\C\end{psmallmatrix}
+s\begin{psmallmatrix}d\\-c\end{psmallmatrix}$
and
$\begin{psmallmatrix}B'\\D'\end{psmallmatrix}=
\begin{psmallmatrix}B\\D\end{psmallmatrix}
+t\begin{psmallmatrix}-d\\c\end{psmallmatrix}$
for some $s\in\p_k$ and $t\in\a_j\p_k$. Thus,
$$
\det\begin{pmatrix}A'&B'\\C'&D'\end{pmatrix}
-\det\begin{pmatrix}A&B\\C&D\end{pmatrix}
=s\det\begin{pmatrix}d&B\\-c&D\end{pmatrix}
+t\det\begin{pmatrix}A&-d\\C&c\end{pmatrix}
=s+t\alpha_k,
$$
so by appropriate choice of $s$ and $t$ we may adjust the determinant by
any element of $\p_k+\alpha_k\a_j\p_k=\o_F$.

Let $\epsilon\in\o_F^\times$ be the unique unit such that
$\epsilon\gamma_1^{-2}\det{M}\in R_2$.
We choose $g=\begin{psmallmatrix}A&B\\C&D\end{psmallmatrix}$
in the above with $\det{g}=\epsilon$.
Then $g\in G_j$ and $\gamma_1^{-1}Mg$ has bottom row
$\begin{pmatrix}\alpha_k&1\end{pmatrix}$. Multiplying on the left
by a suitable $u\in\U_2(F)$, we can make it lower-triangular, i.e.\
$u\gamma_1^{-1}Mg=\begin{psmallmatrix}\gamma_2&\\\alpha_k&1\end{psmallmatrix}$
for some $\gamma_2\in F^\times$. In fact, evaluating the determinant,
we have $\gamma_2=\epsilon\gamma_1^{-2}\det{M}\in R_2$.

Thus,
$uMg=\begin{psmallmatrix}\gamma_1\gamma_2&\\&\gamma_1\end{psmallmatrix}
\begin{psmallmatrix}1&\\\alpha_k&1\end{psmallmatrix}$
is of the required form.
Moreover, although the pair $(u,g)$ is not unique in the above
construction, it is clear that $k$, $\gamma_1$ and $\gamma_2$ are.
This completes the proof.
\end{proof}

\begin{lemma}\label{supp}
We have
$$
W_{\xi_f^0}\!\left(
\begin{pmatrix}\gamma_1\gamma_2\aid_j&&\\&\gamma_1&\\&&1\end{pmatrix}
\begin{pmatrix}1&&\\\aid_j\alpha_k&1&\\&&1\end{pmatrix}\right)
=\psi_{v_k}\big(\alpha_k^{-1}\gamma_2\big)W_{\xi_f^0}\!
\begin{pmatrix}\gamma_1\gamma_2\aid_j\pid_k&&\\
&\gamma_1\pid_k^{-1}&\\&&1\end{pmatrix},
$$
and both sides vanish unless
$\gamma_1\in\p_k$, $\gamma_2\in\a_j^{-1}\p_k^{-2}$.
\end{lemma}
\begin{proof}
Suppose $v\neq v_k$ is a finite place of $F$. Then the corresponding
local factors agree on both sides since $\xi_v^0$ is fixed by
matrices of the form $\begin{psmallmatrix}g&\\&1\end{psmallmatrix}$
for $g\in\GL_2(\o_v)$. Hence, it is enough to prove the local equality
for $v=v_k$. To that end, let us write the $v$-component of $\aid_j\alpha_k$
as $u/\varpi_v$ for $u\in\o_v^\times$. We have
\begin{equation}\label{eqn-supp}
\begin{pmatrix}1&\\u/\varpi_v&1\end{pmatrix}
=\begin{pmatrix}1&u^{-1}\varpi_v\\&1\end{pmatrix}
\begin{pmatrix}\varpi_v&\\&\varpi_v^{-1}\end{pmatrix}
\begin{pmatrix}&-u^{-1}\\u&\varpi_v\end{pmatrix},
\end{equation}
and consequently
$$
W_{\xi_v^0}\!\left(
\begin{pmatrix}\gamma_1\gamma_2\aid_j&&\\&\gamma_1&\\&&1\end{pmatrix}
\begin{pmatrix}1&&\\u/\varpi_v&1&\\&&1\end{pmatrix}\right)
=\psi_v\big(\gamma_2\aid_ju^{-1}\varpi_v\bigr)
W_{\xi_v^0}\!\begin{pmatrix}\gamma_1\gamma_2\aid_j\varpi_v&&\\
&\gamma_1\varpi_v^{-1}&\\&&1\end{pmatrix}.
$$
This concludes the proof of the first assertion.
Next, for any finite place $v$ of $F$ and
$g=\begin{psmallmatrix}a&b\\c&d\end{psmallmatrix}\in\GL_2(F_v)$, one has
$$
W_{\xi_v^0}\!\begin{pmatrix}g&\\&1\end{pmatrix}\neq0
\Longrightarrow c,d\in\o_v.
$$
It follows from this that $\gamma_1\in\p_k$.
On the other hand, for $x\in\o_v$, note that
\begin{align*}
W_{\xi_v^0}\!
\begin{pmatrix}\gamma_1\gamma_2\aid_j\varpi_v&&\\
&\gamma_1\varpi_v^{-1}&\\&&1\end{pmatrix}
&=W_{\xi_v^0}\!\left(
\begin{pmatrix}\gamma_1\gamma_2\aid_j\varpi_v&&\\
&\gamma_1\varpi_v^{-1}&\\&&1\end{pmatrix}
\begin{pmatrix}1&x&\\&1&\\&&1\end{pmatrix}\right)\\
&=\psi_v\big(\gamma_2\aid_j\varpi_v^2x\big)
W_{\xi_v^0}\!
\begin{pmatrix}\gamma_1\gamma_2\aid_j\varpi_v&&\\
&\gamma_1\varpi_v^{-1}&\\&&1\end{pmatrix}.
\end{align*}
Thus, if $W_{\xi_v^0}\!
\begin{psmallmatrix}\gamma_1\gamma_2\aid_j\varpi_v&&\\
&\gamma_1\varpi_v^{-1}&\\&&1\end{psmallmatrix}\neq 0$
then $\gamma_2\aid_j\varpi_v^2\in\o_v$, or in other words
(globally) $\gamma_2\in\a_j^{-1}\p_k^{-2}$.
\end{proof}

In what follows, we use the notation $G^x$ to denote $x^{-1}G x$. Now
\begin{align*}
\Phi_{\xi_\infty,j}\!\begin{pmatrix}h&\beta\\&1\end{pmatrix}
&=\sum_{\gamma\in\U_2(F)\backslash\GL_2(F)}
W_{\xi_\infty}\!\left(\begin{pmatrix}\gamma_\infty&\\&1\end{pmatrix}
\begin{pmatrix}h&\beta\\&1\end{pmatrix}\right)
W_{\xi_f^0}\!\left(\begin{pmatrix}\gamma_f&\\&1\end{pmatrix}g_j\right)\\
&=\sum_{\gamma\in\U_2(F)\backslash\GL_2(F)}
\psi_\infty\big((\gamma_\infty\beta)_2)
W_{\xi_\infty}\!\begin{pmatrix}\gamma_\infty h&\\&1\end{pmatrix}
W_{\xi_f^0}\!\begin{pmatrix}\gamma_fh_j&\\&1\end{pmatrix}\\
&=\sum_{\gamma\in\U_2(F)\backslash\GL_2(F)/G_j}
W_{\xi_f^0}\!\begin{pmatrix}\gamma_fh_j&\\&1\end{pmatrix}
\sum_{\eta\in\U_2(F)^\gamma\cap G_j\backslash G_j}
\psi_\infty\big((\gamma\eta\beta)_2\big)
W_{\xi_\infty}\!\begin{pmatrix}\gamma\eta h&\\&1\end{pmatrix},
\end{align*}
which by Lemmas~\ref{clgp} and \ref{supp} can be written as
\begin{equation}\label{exp2}
\begin{aligned}
\sum_k\frac1{n_{jk}(\q')}
&\sum_{\gamma_1\in R_1,\gamma_2\in R_2}
\psi_{v_k}\big(\alpha_k^{-1}\gamma_2\big)W_{\xi_f^0}\!
\begin{pmatrix}\gamma_1\gamma_2\aid_j\pid_k&&\\&\gamma_1\pid_k^{-1}&\\&&1\end{pmatrix}\\
&\cdot\sum_{\eta\in\U_2(F)^{\ell_k}\cap\Gamma_j(\q')\backslash G_j}
\psi_\infty(\gamma_1(\ell_k\eta\beta)_2)W_{\xi_\infty}\!\left(
\begin{pmatrix}\gamma_1\gamma_2&&\\&\gamma_1&\\&&1\end{pmatrix}
\begin{pmatrix}\ell_k\eta h&\\&1\end{pmatrix}\right),
\end{aligned}
\end{equation}
where $\ell_k=\begin{psmallmatrix}1&\\\alpha_k&1\end{psmallmatrix}$ and
$n_{jk}(\q')=[\U_2(F)^{\ell_k}\cap\Gamma_j(\q'):\U_2(F)^{\ell_k}\cap G_j]$.
One can check that
$$
\U_2(F)^{\ell_k}\cap G_j=
\left\{\ell_k^{-1}\begin{psmallmatrix}1&x\\&1\end{psmallmatrix}\ell_k:
x\in\a_j\p_k^2\right\}.
$$
From this and Lemma~\ref{supp}, it follows that the final
summand in \eqref{exp2}, viewed as a function of $\eta$,
is constant on left cosets of $\U_2(F)^{\ell_k}\cap G_j$ in
$G_j$, so \eqref{exp2} is well defined. Similarly, we have
$$
\U_2(F)^{\ell_k}\cap\Gamma_j(\q')=
\left\{\ell_k^{-1}\begin{psmallmatrix}1&x\\&1\end{psmallmatrix}\ell_k:
x\in\a_j\q'\p_k^2\right\}.
$$
In particular, $n_{jk}(\q')=N(\q')$.

Since $\beta_1\in\a_j\q^{-1}$ and $\beta_2\in\q^{-1}$,
the function $\eta\mapsto\psi_\infty\big(\gamma_1(\ell_k\eta\beta)_2\big)$
is constant on left cosets of $\Gamma_j(\q')$ in $G_j$. Thus \eqref{exp2}
may be rewritten as
\begin{equation}\label{exp3}
\begin{aligned}
\sum_k&\frac1{N(\q')}\sum_{\gamma_1\in R_1,\gamma_2\in R_2}
\psi_{v_k}\big(\alpha_k^{-1}\gamma_2\big)W_{\xi_f^0}\!
\begin{pmatrix}\gamma_1\gamma_2\aid_j\pid_k&&\\&\gamma_1\pid_k^{-1}&\\&&1\end{pmatrix}\\
&\cdot\sum_{\eta\in\Gamma_j(\q')\backslash G_j}
\psi_\infty\big(\gamma_1(\ell_k\eta\beta)_2\big)
\sum_{\epsilon\in\U_2(F)^{\ell_k}\cap\Gamma_j(\q')\backslash\Gamma_j(\q')}
W_{\xi_\infty}\!\left(
\begin{pmatrix}\gamma_1\gamma_2&&\\&\gamma_1&\\&&1\end{pmatrix}
\begin{pmatrix}\ell_k\epsilon\eta h&\\&1\end{pmatrix}\right).
\end{aligned}
\end{equation}
In particular, since we are free to take $\q'=\q$, it follows that
\eqref{phi-j} is left invariant under $\Gamma_j(\q)$.

\subsection{Additive twists}\label{sec-add}
We continue with the above setup. Recall that at every finite place $v$
we have the essential vector $\xi_v^0$ in the space of $\pi_v$ with the
associated essential function $W_{\xi_v^{0}}$ which satisfies
\[
W_{\xi_v^0}\!\begin{pmatrix}1&&\\&1&\\&&1\end{pmatrix}=1.
\]
Let $\varphi\in V_\tau$ be any decomposable vector which is locally the
new vector $\varphi_v^0$ at any finite place $v$ of $F$ and put
$\varphi_j(h)=\varphi(h,h_j), h\in\GL_2(F_\infty)$, so that $\varphi_j\in\mathcal{A}(\Gamma_{0,j}(\q)
\backslash\GL_2(F_\infty),\omega_{\tau_\infty},\chi_\tau^{-1})$. Here again the  essential function $W_{\varphi_{v}^{0}}$ satisfies
\[
W_{\varphi_v^0}\!\begin{pmatrix}1&\\&1\end{pmatrix}=1.
\]
If $\tau_v$ is unramified then \eqref{unr3} with $n=3$ reads as
\begin{equation}\label{unr4}
\int_{\U_2(F_v)\backslash\GL_2(F_v)}
W_{\xi_v^0}\!\begin{pmatrix}g_v&\\&1\end{pmatrix}
W_{\varphi_v^0}(g_v)\,\|\det{g_v}\|_v^{s-\frac{1}{2}}\,dg_v
=L(s,\pi_v\boxtimes\tau_v).
\end{equation}

For each $j=1,\ldots,h$, put
\begin{equation}\label{add}
\Lambda_j(s,\pi,\tau,\xi_\infty,\varphi_\infty,\beta)
=\frac{N(\a_j)^{\frac12-s}}
{[\Gamma_{1,j}(\q):\Gamma_j(\q)]}
\int_{\Gamma_j(\q)\backslash\GL_2(F_\infty)}
\Phi_{\xi_\infty,j}\!\begin{pmatrix}h&{}^t\beta\\&1\end{pmatrix}
\varphi_j(h)\|\det{h}\|_\infty^{s-\frac12}\,dh.
\end{equation}
In view of \eqref{exp3}, this integral is well defined; we call it a
(generalized) \emph{additive twist}.

\begin{lemma}
For $x\in\a_j$, $j=1,\ldots,h$, we have
\[
\Lambda_j(s,\pi,\tau,\xi_\infty,\varphi_\infty,(\beta_1,\beta_2))
=\Lambda_j(s,\pi,\tau,\xi_\infty,\varphi_\infty,(\beta_1+x,\beta_2)).
\]
\end{lemma}
\begin{proof}
Since $U_{\xi_\infty\otimes\xi_f^0}$ is invariant under $Z_3(F)\P_3(F)$
and $\xi_f^0$ is $K_1(\n)$-invariant, it follows that each
$\Phi_{\xi_\infty,j}$ is left invariant by
\[
Z_3(F)\P_3(F)\cap\GL_3(F_\infty)g_jK_1(\n)g_j^{-1}.
\]
From this we see that
\[
\Lambda_j(s,\pi,\tau,\xi_\infty,\varphi_\infty,(\beta_1+x,\beta_2))=\Lambda_j(s,\pi,\tau,\xi_\infty,\varphi_\infty,(\beta_1,\beta_2))
\]
for all $x\in\a_j$.
\end{proof}

If  $(\beta_1,\beta_2)=(0,0)$ then strong approximation implies that
\begin{equation}\label{global}
\sum_j\Lambda_j(s,\pi,\tau,\xi_\infty,\varphi_\infty,\beta)
=\Lambda(s,\pi,\tau,\xi_\infty,\varphi_\infty),
\end{equation}
where
\[
\Lambda(s,\pi,\tau,\xi_\infty,\varphi_\infty)
:=\frac{1}{\vol(K_1(\q))}
\int_{\GL_2(F)\backslash\GL_2(\A_F)}
U_{\xi_\infty\otimes\xi_f^0}\!
\begin{pmatrix}h&\\&1\end{pmatrix}
\varphi(h)\|\det{h}\|^{s-\frac12}\,dh.
\]
If, in addition, $\tau$ is unramified (so that $\q=\o_F$),
then it follows from
\eqref{unr4} that
\begin{equation}\label{unr2}
\Lambda(s,\pi,\tau,\xi_\infty,\varphi_\infty)
=\Psi_\infty(s,\xi_\infty,\varphi_\infty)L(s,\pi\boxtimes\tau)
=\Psi_\infty(s,\xi_\infty,\varphi_\infty)L(s,\pi\times\tau),
\end{equation}
where
$$
L(s,\pi\boxtimes\tau)=\prod_{v<\infty}L(s,\pi_v\boxtimes\tau_v)
\quad\text{and}\quad
L(s,\pi\times\tau)=\prod_{v<\infty}L(s,\pi_v\times\tau_v).
$$

We now state our main result, which is convenient to formulate in terms of
$L(s,\pi\times\tau)$ rather than $L(s,\pi\boxtimes\tau)$.
\begin{theorem}\label{main1}
Let notation be as above, and set $\a=\bigcap_{j=1}^h\a_j$.
Let $\c$ be the conductor of $\chi_\tau$, and choose integral ideals
$\q_1$ and $\q_2$ satisfying
$\c\cap\prod_{\p\mid\q}\p\subseteq\q_2\subseteq\c$ and
$\q\subseteq\q_1\subseteq\prod_{\p\mid\q\q_2^{-1}}\p^{\ord_\p(\q)}$.
Then for any $\beta_2\in\q_2^{-1}\setminus\bigcup_{\p\mid\q_2}\p\q_2^{-1}$,
we have
$$
\tau_\q(\chi_\tau,\beta_2)
\Psi_\infty(s,\xi_\infty,\varphi_\infty)
L(s,\pi\times\tau)
=\frac1{N(\q_1)}\sum_{\beta_1\in\a\q_1^{-1}/\a}\sum_{j=1}^h\Lambda_j(s,\pi,\tau,
\xi_\infty,\varphi_\infty,(\beta_1,\beta_2)),
$$
where
$$
\Psi_\infty(s,\xi_\infty,\varphi_\infty)=
\int_{\U_2(F_\infty)\backslash\GL_2(F_\infty)}W_{\xi_\infty}\!
\begin{pmatrix}h&\\&1\end{pmatrix}
W_{\varphi_\infty}(h)\|\det{h}\|_\infty^{s-\frac12}\,dh
$$
and
$$
\tau_\q(\chi_\tau,\beta_2)=
\sum_{d\in(\o_F/\q)^\times}
\chi_\tau(d)\psi_\infty(d\beta_2).
$$
In particular, if $\q=\c\cap\prod_{\p\mid\q}\p$, then with $\q_2=\q$ we have
$$
\Psi_\infty(s,\xi_\infty,\varphi_\infty)L(s,\pi\times\tau)
=\tau_\q(\chi_\tau,\beta_2)^{-1}
\sum_{j=1}^h\Lambda_j(s,\pi,\tau,\xi_\infty,\varphi_\infty,(0,\beta_2)).
$$
\end{theorem}
\begin{remark}
The condition
$\c\cap\prod_{\p\mid\q}\p\subseteq\q_2\subseteq\c$ ensures that
$\tau_\q(\chi_\tau,\beta_2)\ne0$, by \cite[Corollary~1]{Nemchenok}.
\end{remark}

\subsection{Proof of Theorem \ref{main1}}
Let $\zeta$ be a finite id\`ele with $\zeta_v\in\o_v$ for all
$v<\infty$, $\zeta_v\in\o_v^\times$ whenever $\ord_v(\q)=0$, and
$\zeta=1$ if $\c\ne\q$.
Set $\q'=\q\cap(\zeta)$ and
$h_j'=\begin{psmallmatrix}\aid_j\zeta&\\&1\end{psmallmatrix}$.
We will consider the more general integral
\begin{equation}\label{genadd}
\Lambda_j(s,\pi,\tau,\xi_\infty,\varphi_\infty,\beta,\zeta)
:=\frac{N(\a_j)^{\frac12-s}}
{[\Gamma_{1,j}(\q):\Gamma_j(\q')]}
\int_{\Gamma_j(\q')\backslash\GL_2(F_\infty)}
\Phi_{\xi_\infty,j}\!\begin{pmatrix}h&{}^t\beta\\&1\end{pmatrix}
\varphi(h,h_j')\|\det{h}\|_\infty^{s-\frac12}\,dh,
\end{equation}
where $\beta_1\in\a_j\q_1^{-1}$,
$\beta_2\in\q_2^{-1}\setminus\bigcup_{\p\mid\q_2}\p\q_2^{-1}$.

First, substituting
\eqref{exp3} into \eqref{genadd} and changing $h\mapsto\eta^{-1}h$, we get
\begin{equation}\label{add2}
\begin{aligned}
&N(\a_j)^{\frac12-s}
\frac{N(\q')^{-1}}{[\Gamma_{1,j}(\q):\Gamma_j(\q')]}
\sum_k\sum_{\gamma_1\in R_1,\gamma_2\in R_2}
\psi_{v_k}\big(\alpha_k^{-1}\gamma_2\big)
W_{\xi_f^0}\!
\begin{pmatrix}\gamma_1\gamma_2\aid_j\pid_k&&\\&\gamma_1\pid_k^{-1}&\\&&1\end{pmatrix}\\
&\qquad\cdot\sum_{\eta\in\Gamma_j(\q')\backslash G_j}
\psi_\infty\big(\gamma_1(\ell_k\eta\beta)_2\big)
\int_{\Gamma_j(\q')\backslash\GL_2(F_\infty)}
\sum_{\epsilon\in\U_2(F)^{\ell_k}\cap\Gamma_j(\q')\backslash\Gamma_j(\q')}\\
&\qquad\qquad W_{\xi_\infty}\!\left(
\begin{pmatrix}\gamma_1\gamma_2&&\\&\gamma_1&\\&&1\end{pmatrix}
\begin{pmatrix}\ell_k\epsilon h&\\&1\end{pmatrix}\right)
\varphi\big(\eta^{-1}h,h_j'\big)\|\det{h}\|_\infty^{s-\frac12}\,dh.
\end{aligned}
\end{equation}
Since $h\mapsto\varphi(\eta^{-1}h,h_j')$ is left-invariant under
$\Gamma_j(\q')$,
we may collapse the integral and the sum over $\epsilon$, and make the
change of variable $h\mapsto\ell_k^{-1}h$ to get
\begin{equation}\label{add3}
\begin{aligned}
&N(\a_j)^{\frac12-s}
\frac{N(\q')^{-1}}{[\Gamma_{1,j}(\q):\Gamma_j(\q')]}
\sum_k\sum_{\gamma_1\in R_1,\gamma_2\in R_2}
\psi_{v_k}\big(\alpha_k^{-1}\gamma_2\big)
W_{\xi_f^0}\!
\begin{pmatrix}\gamma_1\gamma_2\aid_j\pid_k&&\\&\gamma_1\pid_k^{-1}&\\&&1\end{pmatrix}\\
&\sum_{\eta\in\Gamma_j(\q')\backslash G_j}
\psi_\infty\big(\gamma_1(\ell_k\eta\beta)_2\big)
\int_{\U_2(F)\cap\Gamma_j(\q')^{\ell_k^{-1}}\backslash\GL_2(F_\infty)}
W_{\xi_\infty}\!\left(
\begin{pmatrix}\gamma_1\gamma_2&&\\&\gamma_1&\\&&1\end{pmatrix}
\begin{pmatrix}h&\\&1\end{pmatrix}\right)\\
&\hspace{5cm}\cdot\varphi\big(\eta^{-1}\ell_k^{-1}h,h_j'\big)
\|\det{h}\|_\infty^{s-\frac12}\,dh.
\end{aligned}
\end{equation}

Let us now consider the archimedean integral in \eqref{add3}.
We change $h\mapsto\gamma_1^{-1}h$ to get
$$
\omega_{\tau_\infty}^{-1}(\gamma_1)\|\gamma_1\|_\infty^{1-2s}
\int_{\U_2(F)\cap\Gamma_j(\q')^{\ell_k^{-1}}\backslash\GL_2(F_\infty)}
W_{\xi_\infty}\!\left(
\begin{pmatrix}\gamma_2&&\\&1&\\&&1\end{pmatrix}
\begin{pmatrix}h&\\&1\end{pmatrix}\right)
\varphi\big(\eta^{-1}\ell_k^{-1}h,h_j'\big)
\|\det{h}\|_\infty^{s-\frac12}\,dh.
$$
Since
$\varphi(\eta^{-1}\ell_k^{-1} h,h_j')=\varphi(h,\ell_k\eta h_j')$,
this becomes
\begin{equation}\label{whit1}
\omega_{\tau_\infty}^{-1}(\gamma_1)\|\gamma_1^2\|_\infty^{\frac12-s}
\int_{\U_2(F_\infty)\backslash\GL_2(F_\infty)}W_{\xi_\infty}\!
\begin{pmatrix}\begin{pmatrix}\gamma_2&\\&1\end{pmatrix}h&\\&1\end{pmatrix}
W_\varphi'(h)\|\det{h}\|_\infty^{s-\frac12}\,dh,
\end{equation}
where
\begin{equation}\label{wprime}
W_\varphi'(h)=
\int_{\U_2(F)\cap\Gamma_j(\q')^{\ell_k^{-1}}\backslash\U_2(F_\infty)}
\varphi(u(x)h,\ell_k\eta h_j')\psi_\infty(\gamma_2 x)\,du(x),
\end{equation}
and $u(x)$ denotes the upper unipotent matrix
$\begin{psmallmatrix}1&x\\&1\end{psmallmatrix}$.
Then it is clear that
the function $h\mapsto W_\varphi'\!\left(
\begin{psmallmatrix}\gamma_2^{-1}&\\&1\end{psmallmatrix}h\right)$
belongs to the Whittaker model of $\tau_\infty$ with respect to the
character $\psi_\infty^{-1}$. Hence, by the local multiplicity one
theorem, there exists a number $a^\varphi_{jk}(\gamma_2,\eta)$ so that
\begin{equation}
\label{mpty}
W_\varphi'(h)=a^\varphi_{jk}(\gamma_2,\eta)
W_{\varphi_\infty}\!\left(
\begin{pmatrix}\gamma_2&\\&1\end{pmatrix}h\right),
\quad\mbox{for all }h\in\GL_2(F_\infty),
\end{equation}
where $W_{\varphi_\infty}$ is the Whittaker function (with respect to
$\psi_\infty^{-1}$) associated to $\varphi_\infty$.  Consequently,
\eqref{whit1} can now be written as
\[
\omega_{\tau_\infty}^{-1}(\gamma_1)a_{jk}^\varphi(\gamma_2,\eta)
\|\gamma_1^2\|_\infty^{\frac12-s}
\int_{\U_2(F_\infty)\backslash\GL_2(F_\infty)}W_{\xi_\infty}\!
\begin{pmatrix}\begin{pmatrix}\gamma_2&\\&1\end{pmatrix}h&\\&1\end{pmatrix}
W_{\varphi_\infty}\!\left(
\begin{pmatrix}\gamma_2&\\&1\end{pmatrix}h\right)
\|\det{h}\|_\infty^{s-\frac12}\,dh,
\]
which after the change
$h\mapsto\begin{psmallmatrix}\gamma^{-1}_2&\\&1\end{psmallmatrix}h$ becomes
\begin{equation}\label{whit2}
\omega_{\tau_\infty}^{-1}(\gamma_1)a_{jk}^\varphi(\gamma_2,\eta)
\|\gamma_1^2\gamma_2\|_\infty^{\frac12-s}\|\gamma_2\|_\infty
\int_{\U_2(F_\infty)\backslash\GL_2(F_\infty)}W_{\xi_\infty}\!
\begin{pmatrix}h&\\&1\end{pmatrix}
W_{\varphi_\infty}(h)\|\det{h}\|_\infty^{s-\frac12}\,dh.
\end{equation}
We note that this is precisely the local Rankin--Selberg integral for
$\pi\times\tau$ at $\infty$, viz.\
$\prod_{v\mid\infty}\Psi_v(s;W_{\xi_v},W_{\varphi_v})$,
in the notation of Section~\ref{prelim}.

\begin{lemma}\label{fc1}
We have
$$
a^\varphi_{jk}(\gamma_2,\eta)=N(\a_j\p_k^2\q')W_{\varphi_f}\!\left(
\begin{pmatrix}\gamma_2&\\&1\end{pmatrix}\ell_k\eta h_j'\right).
$$
\end{lemma}
\begin{proof}
We have $\varphi(g)=\sum_{\gamma\in F^\times}
W_\varphi\!\left(\begin{psmallmatrix}\gamma&\\&1\end{psmallmatrix}g\right)$
for all $g\in\GL_2(\A_F)$, and thus
$$
\varphi(h,\ell_k\eta h_j')
=\sum_{\gamma\in F^\times}
W_{\varphi_f}\!\left(
\begin{pmatrix}\gamma&\\&1\end{pmatrix}\ell_k\eta h_j'\right)
W_{\varphi_\infty}\!\left(\begin{pmatrix}\gamma&\\&1\end{pmatrix}h\right).
$$
Further, a calculation similar to that
in the proof of Lemma~\ref{supp} shows that if
$W_{\varphi_f}\!\left(
\begin{psmallmatrix}\gamma&\\&1\end{psmallmatrix}
\ell_k\eta h_j'\right)$
is non-zero then
$\gamma\in\p_k^{-2}\q_1^{-1}\a_j^{-1}(\zeta^{-1})$.
Plugging this into \eqref{wprime}, we get
\begin{align*}
W_\varphi'(h)&=
\int_{\U_2(F)\cap\Gamma_j(\q')^{\ell_k^{-1}}\backslash\U_2(F_\infty)}
\sum_{\gamma\in F^\times}W_{\varphi_f}\!
\left(\begin{pmatrix}\gamma&\\&1\end{pmatrix}\ell_k\eta h_j'\right)
W_{\varphi_\infty}\!\left(\begin{pmatrix}\gamma&\\&1\end{pmatrix}
u(x)h\right)\psi_\infty(\gamma_2 x)\,du(x)\\
&=\sum_{\gamma\in F^\times}W_{\varphi_f}\!\left(
\begin{pmatrix}\gamma&\\&1\end{pmatrix}\ell_k\eta h_j'\right)
W_{\varphi_\infty}\!\left(\begin{pmatrix}\gamma&\\&1\end{pmatrix}h\right)
\int_{\U_2(F)\cap\Gamma_j(\q')^{\ell_k^{-1}}\backslash\U_2(F_\infty)}
\psi_\infty\big((\gamma_2-\gamma)x\big)\,du(x)\\
&=N(\a_j\p_k^2\q')W_{\varphi_f}\!\left(
\begin{pmatrix}\gamma_2&\\&1\end{pmatrix}\ell_k\eta h_j'\right)
W_{\varphi_\infty}\!\left(\begin{pmatrix}\gamma_2&\\&1\end{pmatrix}h\right).
\end{align*}
Here we have used the fact that
the integral vanishes unless $\gamma=\gamma_2$,
and that $\vol(\U_2(F)\cap\Gamma_j(\q')^{\ell_k^{-1}}\backslash\U_2(F_\infty))
=N(\a_j\p_k^2\q')$, since the additive measure is normalized so that
$\vol(F_\infty/\o_F)=1$.
Our assertion now follows from \eqref{mpty}.
\end{proof}

By \eqref{whit2} and Lemma~\ref{fc1}, we now have
\begin{equation}\label{add4}
\begin{aligned}
\Lambda_j(s,\pi,\tau,\xi_\infty,\varphi_\infty,\beta,\zeta)&=
N(\a_j\p_k^2)N(\a_j)^{\frac12-s}
\prod_{v\mid\infty}\Psi_v(s;W_{\xi_v},W_{\varphi_v})\\
&\qquad\cdot\sum_k\sum_{\gamma_1,\gamma_2\in F^\times/\o_F^\times}
\|\gamma_1^2\gamma_2\|_\infty^{\frac12-s}\|\gamma_2\|_\infty
\psi_{v_k}\big(\alpha_k^{-1}\gamma_2\big)W_{\xi_f^0}\!
\begin{pmatrix}\gamma_1\gamma_2\aid_j\pid_k&&\\&\gamma_1\pid_k^{-1}&\\&&1\end{pmatrix}\\
&\qquad\cdot\frac1{[\Gamma_{1,j}(\q):\Gamma_j(\q')]}
\sum_{\eta\in\Gamma_j(\q')\backslash G_j}
\psi_\infty\big(\gamma_1(\ell_k\eta\beta)_2\big)
W_{\varphi_f}\!\left(\begin{pmatrix}\gamma_1\gamma_2&\\&\gamma_1\end{pmatrix}
\ell_k\eta h_j'\right).
\end{aligned}
\end{equation}

Next we study the average of the inner sum over
$\beta_1\in\a_j\q_1^{-1}/\a_j$:
\begin{equation}\label{innersum}
\frac1{N(\q_1)}\sum_{\beta_1\in\a_j\q_1^{-1}/\a_j}
\frac1{[\Gamma_{1,j}(\q):\Gamma_j(\q')]}\sum_{\eta\in\Gamma_j(\q')\backslash G_j}
\psi_\infty\big(\gamma_1(\ell_k\eta\beta)_2\big)W_{\varphi_f}\!\left(
\begin{pmatrix}\gamma_1\gamma_2&\\&\gamma_1\end{pmatrix}
\ell_k\eta h_j'\right).
\end{equation}
If $\eta=\begin{psmallmatrix}A&B\\C&D\end{psmallmatrix}$ then
$(\ell_k\eta\beta)_2=(A\alpha_k+C)\beta_1+(B\alpha_k+D)\beta_2$.
Since $\gamma_1\alpha_k\in\q$, it follows that
$\psi_\infty\big(\gamma_1(\ell_k\eta\beta)_2\big)
=\psi_\infty(\gamma_1(C\beta_1+D\beta_2))$.

Next note that
$W_{\varphi_f}\!\left(\begin{psmallmatrix}
\gamma_1\gamma_2&\\&\gamma_1\end{psmallmatrix}\ell_k\eta h_j'\right)$
is invariant under replacing $\eta$ by
$\eta\begin{psmallmatrix}1&x\\&1\end{psmallmatrix}$
for $x\in(\zeta)\a_j$, while
$\psi_\infty\big(\gamma_1(\ell_k\eta\beta)_2\big)$ changes by a factor
of $\psi_\infty(C\gamma_1\beta_2x)$.
Hence, averaging over $x\in(\zeta)\a_j/\q'\a_j$, we see that \eqref{innersum}
vanishes unless $C\gamma_1\in\q_2(\zeta^{-1})\a_j^{-1}$.
Further, computing the average over $\beta_1$, we get $0$ unless
$C\gamma_1\in\q_1\a_j^{-1}$. Therefore, we can restrict $\eta$ to
$\Gamma_{0,j}(\m)$, where $\m=\q_1(\gamma_1^{-1})\cap
\q_2(\zeta^{-1})(\gamma_1^{-1})\cap\o_F$.
Since $\Gamma_j(\q')$ is normal in $\Gamma_{0,j}(\m)$, we may swap left and
right cosets, so \eqref{innersum} becomes
\begin{equation}\label{innersum2}
\frac1{[\Gamma_{1,j}(\q):\Gamma_j(\q')]}
\sum_{\eta\in\Gamma_{0,j}(\m)/\Gamma_j(\q')}
\psi_\infty(\gamma_1\beta_2\eta_{22})W_{\varphi_f}\!\left(
\begin{pmatrix}\gamma_1\gamma_2&\\&\gamma_1\end{pmatrix}
\ell_k\eta h_j'\right),
\end{equation}
where $\eta_{22}$ denotes the lower-right entry of $\eta$.

Now set $\m'=\q_1\cap\q_2(\zeta^{-1})$,
$$
X_0=\left\{\begin{psmallmatrix}a&b\\c&d\end{psmallmatrix}\in\Gamma_{0,j}(\m'):
b\in(\zeta)\a_j\right\},
\quad
X_1=\left\{\begin{psmallmatrix}a&b\\c&d\end{psmallmatrix}\in X_0:
d-1\in\q\right\}.
$$
Then as a function of $\eta$, \eqref{innersum2} is right invariant under
$X_1$ and acts via the nebentypus character $\chi_\tau$ under $X_0$.
Hence, \eqref{innersum2} equals
\begin{equation}\label{innersum3}
\frac{[X_1:\Gamma_j(\q')]}{[\Gamma_{1,j}(\q):\Gamma_j(\q')]}
\sum_{\mu\in\Gamma_{0,j}(\m)/X_0}
W_{\varphi_f}\!\left(
\begin{pmatrix}\gamma_1\gamma_2&\\&\gamma_1\end{pmatrix}\ell_k\mu h_j'\right)
\sum_{\nu\in X_0/X_1}\chi_\tau(\nu)
\psi_\infty(\gamma_1\beta_2\mu_{22}\nu_{22}).
\end{equation}

Next, note that for a fixed $\gamma_1$,
the ideal $\gamma_1\mu_{22}\o_F+\q$ is independent of the choice of
representative for $\mu$. We will show that the sum over those $\mu$ for which
$\gamma_1\mu_{22}\o_F+\q\ne\o_F$ vanishes.
First, if $\gamma_1\mu_{22}\o_F+\c\ne\o_F$ then the Gauss sum
$$
\sum_{\nu\in X_0/X_1}\chi_\tau(\nu)
\psi_\infty(\gamma_1\beta_2\mu_{22}\nu_{22})
=\sum_{d\in(\o_F/\q)^\times}\chi_\tau(d)
\psi_\infty(\gamma_1\beta_2\mu_{22}d)
$$
vanishes. Hence, we may assume that $(\zeta)=\o_F$, so
$\m'=\q$ and $X_0=\Gamma_{0,j}(\q)$.

If $\gamma_1\o_F+\q=\o_F$ then $\m=\q$, whence $D\o_F+\q=\o_F$, so
we may assume that $\gamma_1\o_F\subseteq\p$ for some prime ideal
$\p\supseteq\q$. We split the coset representatives for $\mu$ as
$g_1g_2$, where
$g_1=\begin{psmallmatrix}A_1&B_1\\C_1&D_1\end{psmallmatrix}$
runs through representatives for
$\Gamma_{0,j}(\m)/\Gamma_{0,j}(\p^{-1}\q)$ and
$g_2=\begin{psmallmatrix}A_2&B_2\\C_2&D_2\end{psmallmatrix}$
runs through representatives for
$\Gamma_{0,j}(\p^{-1}\q)/\Gamma_{0,j}(\q)$.
Then
$\psi_\infty(\gamma_1\mu_{22}\beta_2)
=\psi_\infty(\gamma_1D_1D_2\beta_2)$.
Defining
$$
f_{\gamma_1,g_1}(g_2)=
\sum_{d\in(\o_F/\q_2)^\times}\chi_\tau(d)
\psi_\infty(\gamma_1D_1D_2d\beta_2),
$$
it is plain that
$f_{\gamma_1,g_1}(kg_2k')=f_{\gamma_1,g_1}(g_2)\overline{\chi_\tau(k')}$
for any $k\in\Gamma_{1,j}(\p^{-1}\q)$,
$k'\in\Gamma_{0,j}(\q)$.
As the following lemma shows, these terms therefore contribute nothing
to \eqref{innersum3}.

\begin{lemma}
Suppose $\p$ is prime ideal dividing $\q$.
Let $f:\Gamma_{0,j}(\p^{-1}\q)\to\C$ be left invariant under
$\Gamma_{1,j}(\p^{-1}\q)$ and transform by
$\overline{\chi}_\tau$ on the right under $\Gamma_{0,j}(\q)$.
Then for any $x\in\GL_2(\A_{F,f})$,
$$
\sum_{g\in\Gamma_{0,j}(\p^{-1}\q)/\Gamma_{0,j}(\q)}W_{\varphi_f}(xgh_j)f(g)=0.
$$
\end{lemma}
\begin{proof}
Let $v$ be the place corresponding to $\p$, and put
$m=\ord_\p(\q)$.
For any subgroup $G\subseteq\GL_2$, let $G^1$ denote $G\cap\SL_2$.
The natural map
$$
\Gamma_{1,j}^1(\p^{-1}\q)\to K_{1,v}^1(\p^{m-1})/K_v^1(\p^m)
$$
is surjective by strong approximation for $\SL_2(\A_F)$.
Fix $k_1\in K_{1,v}^1(\p^{m-1})$ and choose
$k\in\Gamma_{1,j}^1(\p^{-1}\q)$ that maps to the coset $k_1K_v^1(\p^m)$
above. We have
$$
\sum_{g\in\Gamma_{0,j}(\p^{-1}\q)/\Gamma_{0,j}(\q)}W_{\varphi_f}(xgh_j)f(g)
=\sum_{g\in\Gamma_{0,j}(\p^{-1}\q)/\Gamma_{0,j}(\q)}W_{\varphi_f}(xkgh_j)f(g).
$$

For any finite place $w\ne v$, we have
$$
W_{\varphi_w}(xkgh_j)=W_{\varphi_w}(xgh_j),
$$
since $(gh_j)^{-1}k_w(gh_j)\in K_{1,w}(\p_w^{\ord_{\p_w}(\q)})$.
At $v$ we have $k_v=k_1k'$ for some $k'\in K_v^1(\p^m)$, so that
$$
W_{\varphi_v}(xkgh_j)=W_{\varphi_v}(xk_1k'g)
=W_{\varphi_v}(xk_1gg^{-1}k'g)
=W_{\varphi_v}(xk_1g)
=W_{\varphi_v}(xgg^{-1}k_1g),
$$
since $g^{-1}k'g\in g^{-1}K_v^1(\p^m)g=K_v^1(\p^m)$.

Now, if $k_1$ runs through a set of
representatives for $K_{1,v}^1(\p_v^{m-1})/K_v^1(\p_v^m)$,
so does $g^{-1}k_1g$, since
$g$ normalizes both $K_{1,v}^1(\p_v^{m-1})$ and $K_v^1(\p_v^m)$.
Thus, averaging over $k_1$, our sum becomes
$$
\sum_{g\in\Gamma_{0,j}(\p^{-1}\q)/\Gamma_{0,j}(\q)}f(g)
\prod_{\substack{w<\infty\\w\ne v}}W_{\varphi_w}(xgh_j)
\cdot\frac1{[K_{1,v}^1(\p_v^{m-1}):K_{1,v}^1(\p_v^m)]}
\sum_{k_1}W_{\varphi_v}(xgk_1).
$$

Consider the inner sum
$\sum_{k_1}W_{\varphi_v}(tk_1)$ as a function of $t$.
Fixing $k\in K_{1,v}(\p_v^{m-1})$, we have $k\in
\begin{psmallmatrix}y&\\&1\end{psmallmatrix}K_{1,v}^1(\p_v^{m-1})$
for some $y\in\o_v^\times$. Hence,
$$
\sum_{k_1}W_{\varphi_v}(tkk_1)
=\sum_{k_1}W_{\varphi_v}(t\begin{psmallmatrix}y&\\&1\end{psmallmatrix}k_1)
=\sum_{k_1}W_{\varphi_v}(tk_1\begin{psmallmatrix}y&\\&1\end{psmallmatrix})
=\sum_{k_1}W_{\varphi_v}(tk_1),
$$
since $\begin{psmallmatrix}y&\\&1\end{psmallmatrix}\in K_{1,v}(\p_v^m)$.
Therefore $\sum_{k_1}W_{\varphi_v}(tk_1)=0$.
\end{proof}

Hence, we may assume that $\gamma_1\mu_{22}\o_F+\q=\o_F$.
In this case, we can choose representatives
of the form $\mu=\begin{psmallmatrix}1&x\\&1\end{psmallmatrix}$,
where $x\in\a_j\p_k^2/(\zeta)\a_j\p_k^2$, and we have
$$
\sum_{\nu\in X_0/X_1}\chi_\tau(\nu)
\psi_\infty(\gamma_1\beta_2\mu_{22}\nu_{22})
=\overline{\chi_\tau(\gamma_1)}
\tau_\q(\chi_\tau,\beta_2).
$$
Further, we compute that
$$
W_{\varphi_f}\!\left(
\begin{pmatrix}\gamma_1\gamma_2&\\&\gamma_1\end{pmatrix}\ell_k\eta h_j'\right)
=\omega_{\tau_\infty}^{-1}(\gamma_1)W_{\varphi_f}\!\left(
\begin{pmatrix}\gamma_2&\\&1\end{pmatrix}\ell_kh_j'\right).
$$

Consequently the finite part of
$\frac1{N(\q_1)}\displaystyle\sum_{\beta_1\in\a\q_1^{-1}/\a}\Lambda_j(s,\pi,\tau,\xi_\infty,\varphi_\infty,\beta,\zeta)$
takes the form
\begin{align*}
&\frac{N(\zeta)[X_1:\Gamma_j(\q')]}{[\Gamma_{1,j}(\q):\Gamma_j(\q')]}
N(\a_j)^{\frac12-s}
\sum_kN(\a_j\p_k^2)\|\gamma_2\|_\infty\\
&\cdot\sum_{\gamma_1\in\o_F^\times\backslash\p_k\cap F^\times}
\sum_{\gamma_2\in\o_F^\times\backslash\a_j^{-1}\p_k^{-2}\cap F^\times}
\|\gamma_1^2\gamma_2\|_\infty^{\frac12-s}W_{\xi_f^0}\!
\begin{pmatrix}\gamma_1\gamma_2\aid_j\pid_k&&\\&\gamma_1\pid_k^{-1}&\\&&1\end{pmatrix}\\
&\cdot\tau_\q(\chi_\tau,\beta_2)
\psi_{v_k}\big(\alpha_k^{-1}\gamma_2\big)
W_{\varphi_f}\!\left(
\begin{pmatrix}\gamma_2&\\&1\end{pmatrix}\ell_kh_j'\right)
\omega_{\tau_\infty}^{-1}(\gamma_1)
\overline{\chi_\tau(\gamma_1)}.
\end{align*}
For $\gamma_1,\gamma_2$ as above, let $\m_1=(\gamma_1)\p_k^{-1}$
and $\m_2=(\gamma_2)\a_j\p_k^2$. Then
$\|\gamma_1^2\gamma_2\|_\infty=N(\m_1^2\m_2\a_j^{-1})$
and $(\gamma_1)=\p_k\m_1$.
\begin{lemma}
We have
$$
\psi_{v_k}\big(\alpha_k^{-1}\gamma_2\big)
W_{\varphi_f}\!\left(
\begin{pmatrix}\gamma_2\pid_k&\\&\pid_k\end{pmatrix}\ell_kh_j'\right)
=W_{\varphi_f}\!\begin{pmatrix}\zeta&\\&1\end{pmatrix}
W_{\varphi_f}\!\left(
\begin{pmatrix}\gamma_2\pid_k^2&\\&1\end{pmatrix}h_j\right).
$$
\end{lemma}

\begin{proof}
We check this locally at every place $v$ by the same computation as in
the proof of Lemma~\ref{supp}. The proof here splits naturally into two cases. First, suppose $\c\neq \q$. Then $\zeta=1$ and the lemma is trivial to check at any finite place $v\neq v_k$. For $v=v_k$, we write the $v$-component of $\alpha_k\aid_j=u/\varpi_v$ as in Lemma~\ref{supp} and use the decomposition (\ref{eqn-supp}) to re-write the left hand side (locally at $v$) as
\begin{align*}
\psi_{v_k}\big(\alpha_k^{-1}\gamma_2\big)
W_{\varphi_v^0}\!\left(
\begin{pmatrix}\gamma_2\varpi_v&\\&\varpi_v\end{pmatrix}\ell_kh_j'\right)
&=\psi_{v_k}\big(\alpha_k^{-1}\gamma_2\big)
W_{\varphi_v^0}\!\left(
\begin{pmatrix}\gamma_2\aid_j\varpi_v&\\&\varpi_v\end{pmatrix}
\begin{pmatrix}1&\\u/\varpi_v&1\end{pmatrix}\right)\\
&=W_{\varphi_v^0}\!\left(
\begin{pmatrix}\gamma_2\aid_j\varpi^2_v&\\&1\end{pmatrix}
\begin{pmatrix}0&-u^{-1}\\u&\varpi_v\end{pmatrix}\right).
\end{align*}
Since the prime ideal $\p_k$ corresponding to the place $v=v_k$
is co-prime to $\q$, we see that $W_{\varphi_v^{0}}$ is right
$\GL_2(\o_v)$-invariant and thus the conclusion in the case at hand.

Next, suppose $\c=\q$. For $v<\infty$ such that $\ord_v(\q)=0$,
since by choice $\zeta_v\in\o_v^\times$, it follows that
$W_{\varphi_v^0}$ is right invariant under translation by
$\begin{psmallmatrix}\zeta_v&\\&1\end{psmallmatrix}$. Consequently the
proof for such $v$ is similar to the above argument. Henceforth we
assume $v$ is such that $\ord_v(\q)>0$. Then $v\neq v_k$, and we need
to show
\[
W_{\varphi_v^0}\!\left(
\begin{pmatrix}\gamma_2&\\&1\end{pmatrix}\ell_kh_{j,v}'\right)
=W_{\varphi_v^0}\!\begin{pmatrix}\zeta_v&\\&1\end{pmatrix}
W_{\varphi_v^0}\!\left(
\begin{pmatrix}\gamma_2&\\&1\end{pmatrix}h_{j,v}\right).
\]
Moving $h_{j,v}'$ past $\ell_k$ to the left, we re-write the left hand side of the above equation as
\[
W_{\varphi_v^0}\!\left(
\begin{pmatrix}\gamma_2\zeta_v\aid_{j,v}&\\&1\end{pmatrix}
\begin{pmatrix}1&\\\alpha_k\aid_{j,v}\zeta_v&1\end{pmatrix}\right),
\]
which in turn equals
\[
W_{\varphi_v^0}\!\begin{pmatrix}\gamma_2\zeta_v\aid_{j,v}&\\&1\end{pmatrix},
\]
since $\ord_v(\alpha_k\aid_{j,v}\zeta_v)\geq\ord_v(\q)$ by choice of
$\alpha_k$. It therefore remains to prove
\begin{equation}\label{comp-mult}
W_{\varphi_v^0}\!\begin{pmatrix}\gamma_2\zeta_v\aid_{j,v}&\\&1\end{pmatrix}
=W_{\varphi_v^0}\!\begin{pmatrix}\zeta_v&\\&1\end{pmatrix}
W_{\varphi_v^0}\!\begin{pmatrix}\gamma_2\aid_{j,v}&\\&1\end{pmatrix}.
\end{equation}

Let $\text{St}$ denote the Steinberg representation of $\GL_2(F_v)$,
i.e., the unique Langlands quotient of the induced module
$\Ind(\|\cdot\|_v^{-1/2},\|\cdot\|_v^{1/2})$. It is a well-known fact
that $\tau_v$ is either supercuspidal, or a twist of the Steinberg
representation $\text{St}\otimes\chi$, or an irreducible principal series
representation $\Ind(\chi_1,\chi_2)$, where $\chi,\chi_1$ and $\chi_2$
are characters of $F_v^\times$. Since $\ord_v(\c)=\ord_v(\q)>0$,
it follows from \cite[Proposition~3.4]{Tu} that $\tau_v$ cannot be
supercuspidal. Further, using the formula for the conductor of $\tau_v$
in the remaining cases (see \cite[Remark 4.25]{Ge}), we conclude that
$\tau_v\cong\Ind(\chi_1,\chi_2)$ with either $\chi_1$ or $\chi_2$ (but
not both) ramified. Thus one of the Langlands parameters of $\tau_v$
is zero; say $(\alpha,\beta)=(\alpha,0)$ are the Langlands parameters
of $\tau_v$. Now, we may appeal to the explicit description of the
essential function $W=W_{\varphi_v^0}$ (cf.~\cite[Theorem~4.1]{Miy})
which in our case ($n=2$) reads as
\[
W\!\begin{pmatrix}\varpi_v^f&\\&1\end{pmatrix}
=\begin{cases}
q_v^{-f/2}s_f(\alpha,\beta)&\text{if }f\geq0,\\
0&\mbox{otherwise},
\end{cases}
\]
where $s_f(\alpha,\beta)$ is the Schur polynomial
$\frac{\alpha^{f+1}-\beta^{f+1}}{\alpha-\beta}$. Since
$\beta=0$ and $\ord_v(\zeta_v),\ord_v(\gamma_2\aid_{j,v})\geq0$,
our assertion \eqref{comp-mult} follows.
\end{proof}

We define
\begin{equation}\label{ddc}
\lambda_\pi(\m_1,\m_2)=
N(\m_1\m_2)W_{\xi_f^0}\!
\begin{pmatrix}\idm_1\idm_2&&\\&\idm_1&\\&&1\end{pmatrix}
\end{equation}
and
$$
\lambda_\tau(\m_2)=\sqrt{N(\m_2)}W_{\varphi_f}\!
\begin{pmatrix}\idm_2&\\&1\end{pmatrix},
$$
where $\idm_1$ and $\idm_2$ are finite id\`eles such that
$\m_1=(\idm_1)$ and $\m_2=(\idm_2)$. Then
\begin{align*}
\frac1{N(\q_1)}\sum_{\beta_1\in\a\q_1^{-1}/\a}
\Lambda_j(s,\pi,\tau,\xi_\infty,\varphi_\infty,\beta,\zeta)&=
\frac{\sqrt{N(\zeta)}[X_1:\Gamma_j(\q')]}{[\Gamma_{1,j}(\q):\Gamma_j(\q')]}
\tau_\q(\chi_\tau,\beta_2)
\lambda_\tau((\zeta))
\prod_{v\mid\infty}\Psi_v(s;W_{\xi_v},W_{\varphi_v})\\
&\qquad\cdot\sum_k\sum_{\m_1\sim\p_k^{-1}}\sum_{\m_2\sim\a_j\p_k^2}
\lambda_\pi(\m_1,\m_2)\lambda_\tau(\m_2)\chi_{\omega_\tau}(\m_1)
N(\m_1^2\m_2)^{-s},
\end{align*}
where $\chi_{\omega_\tau}$ is the \gr{} of modulus $\q$ associated to
$\omega_\tau$ (which need not be primitive).
Noting that
$\frac{[X_1:\Gamma_j(\q')]}{[\Gamma_{1,j}(\q):\Gamma_j(\q')]}
=\frac{N(\q)}{N(\q_1(\zeta)\cap\q_2)}$,
we can now sum over $j$ to get
\begin{align*}
\frac1{N(\q_1)}\sum_{\beta_1}\sum_j\Lambda_j(s,\pi,\tau,\xi_\infty,\varphi_\infty,\beta,\zeta)
&=\frac{\sqrt{N(\q^2(\zeta))}}{N(\q_1(\zeta)\cap\q_2)}
\tau_\q(\chi_\tau,\beta_2)
\lambda_\tau((\zeta))
\prod_{v\mid\infty}\Psi_v(s;W_{\xi_v},W_{\varphi_v})\\
&\qquad\cdot\sum_{\m_1,\m_2}
\lambda_\pi(\m_1,\m_2)\lambda_\tau(\m_2)\chi_{\omega_\tau}(\m_1)
N(\m_1^2\m_2)^{-s}.
\end{align*}
It remains to identify the Dirichlet series in the above expression.
\begin{lemma}\label{dd}
We have
\begin{equation}\label{e:doublesum}
\sum_{\m_1,\m_2}\lambda_\pi(\m_1,\m_2)\lambda_\tau(\m_2)
\chi_{\omega_\tau}(\m_1)N(\m_1^2\m_2)^{-s}=L(s,\pi\times\tau).
\end{equation}
\end{lemma}
\begin{proof}
For a fixed choice of $\tau$ and a non-zero integral ideal $\a$, define
$$
c_{\pi,\tau}(\a)=\sum_{\m_1^2\m_2=\a}\lambda_\pi(\m_1,\m_2)
\lambda_\tau(\m_2)\chi_{\omega_\tau}(\m_1).
$$
Then, for any unramified id\`ele class character $\omega$, we have
$$
\lambda_{\tau\otimes\omega}(\m_2)
=\lambda_\tau(\m_2)\chi_\omega(\m_2)
\quad\text{and}\quad
\chi_{\omega_{\tau\otimes\omega}}(\m_1)
=\chi_{\omega_\tau}(\m_1)\chi_\omega(\m_1)^2,
$$
so that
\begin{equation}\label{e:ctwist}
c_{\pi,\tau\otimes\omega}(\a)=c_{\pi,\tau}(\a)\chi_\omega(\a).
\end{equation}
Next define $\lambda_{\pi\times\tau}$ to be the Dirichlet coefficients
of $L(s,\pi\times\tau)$, so that, for any unramified $\omega$, we have
\begin{equation}\label{e:lambdatwist}
L(s,\pi\times(\tau\otimes\omega))=
\sum_\a\lambda_{\pi\times\tau}(\a)\chi_\omega(\a)N(\a)^{-s}.
\end{equation}
Note that both $c_{\pi,\tau}$ and $\lambda_{\pi\times\tau}$ are
multiplicative, so it suffices to show that they agree at prime powers.

Given integers
\begin{equation}\label{e:schurlambda}
\lambda_1\ge\lambda_2\ge\lambda_3\ge0,
\end{equation}
the Schur
polymial $s_{\lambda_1,\lambda_2,\lambda_3}(x_1,x_2,x_3)$ is the ratio
\begin{equation}\label{e:schurpoly}
\frac{\det\begin{pmatrix}
x_1^{\lambda_1+2}&x_2^{\lambda_1+2}&x_3^{\lambda_1+2}\\
x_1^{\lambda_2+1}&x_2^{\lambda_2+1}&x_3^{\lambda_2+1}\\
x_1^{\lambda_3}&x_2^{\lambda_3}&x_3^{\lambda_3}
\end{pmatrix}}
{\det\begin{pmatrix}
x_1^2&x_2^2&x_3^2\\
x_1&x_2&x_3\\
1&1&1
\end{pmatrix}}.
\end{equation}
By the Cauchy identity \cite[Theorem~43.3]{bump}, for sufficiently small
$\alpha_1,\alpha_2,\alpha_3,\gamma_1,\gamma_2,\gamma_3\in\C$,
we have
$$
\prod_{i=1}^3\prod_{j=1}^3\frac1{1-\alpha_i\gamma_j}
=\sum_\lambda s_\lambda(\alpha_1,\alpha_2,\alpha_3)
s_\lambda(\gamma_1,\gamma_2,\gamma_3),
$$
where the sum runs over all $\lambda=(\lambda_1,\lambda_2,\lambda_3)$
satisfying \eqref{e:schurlambda}.
Note that if $\gamma_3=0$ then only
terms with $\lambda_3=0$ contribute to the above sum. Writing
$\lambda_2=k_1$, $\lambda_1=k_1+k_2$ and replacing
$(\gamma_1,\gamma_2,\gamma_3)$ by $(x\gamma_1,x\gamma_2,0)$ for a small
$x\in\C$, we obtain
$$
\prod_{i=1}^3\prod_{j=1}^2\frac1{1-\alpha_i\gamma_jx}
=\sum_{k_1=0}^\infty\sum_{k_2=0}^\infty
s_{k_1+k_2,k_1,0}(\alpha_1,\alpha_2,\alpha_3)
s_{k_1+k_2,k_1,0}(\gamma_1,\gamma_2,0)x^{2k_1+k_2}.
$$

Now fix a prime ideal $\p$, and let $v$ be the corresponding place of
$F$; then $q_v=N(\p)$. Let $\{\alpha_1,\alpha_3,\alpha_3\}$ (resp.\
$\{\gamma_1,\gamma_2\}$) denote the Langlands
parameters of $\pi_v$ (resp.\ $\tau_v$), so that
\begin{equation}\label{e:pitaueuler}
L(s,\pi_v)=\prod_{i=1}^3\frac1{1-\alpha_iN(\p)^{-s}}
\quad\text{and}\quad
L(s,\tau_v)=\prod_{j=1}^2\frac1{1-\gamma_jN(\p)^{-s}}.
\end{equation}
Then by the above we have
\begin{align*}
L(s,\pi_v\times\tau_v)
&=\prod_{i=1}^3\prod_{j=1}^2\frac1{1-\alpha_i\gamma_jN(\p)^{-s}}\\
&=\sum_{k_1=0}^\infty\sum_{k_2=0}^\infty
s_{k_1+k_2,k_1,0}(\alpha_1,\alpha_2,\alpha_3)
s_{k_1+k_2,k_1,0}(\gamma_1,\gamma_2,0)N(\p)^{-(2k_1+k_2)s}.
\end{align*}
On the other hand, we have
$$
L(s,\pi_v\times\tau_v)
=\sum_{k=0}^\infty\lambda_{\pi\times\tau}(\p^k)N(\p)^{-ks},
$$
whence
\begin{equation}\label{e:rscauchy}
\lambda_{\pi\times\tau}(\p^k)
=\sum_{2k_1+k_2=k}
s_{k_1+k_2,k_1,0}(\alpha_1,\alpha_2,\alpha_3)
s_{k_1+k_2,k_1,0}(\gamma_1,\gamma_2,0).
\end{equation}
Further, by \eqref{e:pitaueuler} and the identity
$L(s,\tau_v)=\sum_{k=0}^\infty\lambda_\tau(\p^k)N(\p)^{-ks}$,
we have
$$
\lambda_\tau(\p^{k_2})=\sum_{j=0}^{k_2}\gamma_1^j\gamma_2^{k_2-j}
=\frac{\gamma_1^{k_2+1}-\gamma_2^{k_2+1}}{\gamma_1-\gamma_2}.
$$
Moreover, $\chi_{\omega_\tau}(\p^{k_1})=(\gamma_1\gamma_2)^{k_1}$,
so that
\begin{equation}\label{e:gl2coeff}
\lambda_\tau(\p^{k_2})\chi_{\omega_\tau}(\p^{k_1})
=(\gamma_1\gamma_2)^{k_1}\frac{\gamma_1^{k_2+1}-\gamma_2^{k_2+1}}{\gamma_1-\gamma_2}
=s_{k_1+k_2,k_1,0}(\gamma_1,\gamma_2,0),
\end{equation}
by \eqref{e:schurpoly}. Thus,
\begin{equation}\label{e:ctau}
c_{\pi,\tau}(\p^k)=\sum_{2k_1+k_2=k}
\lambda_\pi(\p^{k_1},\p^{k_2})
s_{k_1+k_2,k_1,0}(\gamma_1,\gamma_2,0).
\end{equation}

Suppose now that $\tau$ is unramified. Then we have $\q_1=\o_F$, so we may take
$\beta_1=\beta_2=0$. In that case,  we have (cf.\ \eqref{unr2})
$$
\sum_j\Lambda_j(s,\pi,\tau,\xi_\infty,\varphi_\infty,\beta)
=\Psi_\infty(s,\xi_\infty,\varphi_\infty)L(s,\pi\times\tau).
$$
Choosing $\xi_\infty$ and $\varphi_\infty$ such that
$\Psi_\infty(s,\xi_\infty,\varphi_\infty)\ne0$, we conclude that
\eqref{e:doublesum} holds. Replacing $\tau$ by an unramified twist
$\tau\otimes\omega$, by \eqref{e:ctwist} and \eqref{e:lambdatwist}, we
have
$$
\sum_\a c_{\pi,\tau}(\a)\chi_\omega(\a)N(\a)^{-s}
=\sum_\a\lambda_{\pi\times\tau}(\a)\chi_\omega(\a)N(\a)^{-s}
$$
for any unramified $\omega$. By \cite[Lemma~4.2]{B-Kr4} it follows that
$c_{\pi,\tau}(\a)=\lambda_{\pi\times\tau}(\a)$ for all $\a$.
In particular, taking $\a=\p^k$, by
\eqref{e:rscauchy} and \eqref{e:ctau} we find that
\begin{equation}\label{e:rscoeffs}
\sum_{2k_1+k_2=k}
s_{k_1+k_2,k_1,0}(\alpha_1,\alpha_2,\alpha_3)
s_{k_1+k_2,k_1,0}(\gamma_1,\gamma_2,0)
=\sum_{2k_1+k_2=k}
\lambda_\pi(\p^{k_1},\p^{k_2})
s_{k_1+k_2,k_1,0}(\gamma_1,\gamma_2,0),
\end{equation}
whenever $\tau$ is unramified. Applying this with
$\tau=\|\cdot\|^{it_1}\boxplus\|\cdot\|^{it_2}$
for arbitrary $t_1,t_2\in\R$, the Satake parameters
$(\gamma_1,\gamma_2)=(N(\p)^{-it_1},N(\p)^{-it_2})$ are Zariski-dense in
$\C^2$, so \eqref{e:rscoeffs} holds for arbitrary $\gamma_1,\gamma_2\in\C$.
Further, from \eqref{e:gl2coeff} it is easy to see that
the polynomials $s_{k_1+k_2,k_1,0}(x_1,x_2,0)$, for $k_1,k_2$ ranging
over all non-negative integers, are linearly independent.
Therefore, from \eqref{e:rscoeffs} we conclude that
$\lambda_\pi(\p^{k_1},\p^{k_2})=s_{k_1+k_2,k_1,0}(\alpha_1,\alpha_2,\alpha_3)$.

Finally, applying this together with \eqref{e:rscauchy} and
\eqref{e:ctau} for an arbitrary $\tau$ (not necessarily
unramified), we conclude that
$c_{\pi,\tau}(\p^k)=\lambda_{\pi\times\tau}(\p^k)$, as desired.
\end{proof}
Hence, we obtain
\begin{equation}\label{main1res}
\begin{aligned}
\frac1{N(\q_1)}\sum_{\beta_1}\sum_j\Lambda_j(s,\pi,\tau,\xi_\infty,\varphi_\infty,\beta,\zeta)
&=\frac{\sqrt{N(\q^2(\zeta))}}{N(\q_1(\zeta)\cap\q_2)}
\tau_\q(\chi_\tau,\beta_2)
\lambda_\tau((\zeta))
\Psi_\infty(s,\xi_\infty,\varphi_\infty)
L(s,\pi\times\tau).
\end{aligned}
\end{equation}
Taking $\zeta=1$, we obtain the conclusion
of Theorem~\ref{main1}.

\subsection{Functional equation}\label{fe}
In this section we assume both $\pi$ and $\tau$ are cuspidal automorphic representations. We then show how to use Theorem~\ref{main1} to give a new
proof of the analytic properties of $L(s,\pi\boxtimes\tau)$ when $\pi$
and $\tau$ have coprime conductors and $\tau_v$ is a twist-minimal
principal series representation for all finite $v$.

Let $\xi=\bigotimes_v\xi_v$ be a pure tensor in the space of $\pi$
such that $\xi_v=\xi_v^0$ for every finite $v$. Let
$(U_\xi,U_\xi')$
be the \emph{automorphic pair} attached to such a $\xi$ as in
Lemma~\ref{aut-dual}. Likewise, let $(U_\varphi,U_\varphi')$ be the
automorphic pair associated with a pure tensor $\bigotimes_v\varphi_v$
in the space of $\tau$ satisfying $\varphi_v=\varphi_v^0$ for all
$v<\infty$.

\begin{theorem}\label{main2}
Keep the above notation as well as that of Theorem~\ref{main1}. Assume
that $\pi$ and $\tau$ are cuspidal automorphic, with $\n+\q=\o_F$ and
$\c=\q$. Then $L(s,\pi\boxtimes\tau)$ continues to an entire function
and satisfies the functional equation
$$
\Psi_\infty(s,\xi_\infty,\varphi_\infty)L(s,\pi\boxtimes\tau)
=\epsilon N(\n^2\q^3)^{\frac12-s}
\widetilde{\Psi}_\infty(1-s,\xi_\infty,\varphi_\infty)
L(1-s,\widetilde{\pi}\boxtimes\widetilde{\tau}),
$$
where
$$
\widetilde{\Psi}_\infty(s,\xi_\infty,\varphi_\infty)=
\prod_{v\mid\infty}\Psi_v(s;\widetilde{W}_{\tilde{\xi}_v},\widetilde{W}_{\tilde{\varphi}_v})
\quad\mathrm{(cf.\ \eqref{tildedef})}
$$
and $\epsilon$ is as defined in \eqref{epsdef}.
\end{theorem}
\begin{proof}
For any finite-index subgroup $H\le\Gamma_j(\q)$, we have
\begin{align*}
&\Lambda_j(s,\pi,\tau,\xi_\infty,\varphi_\infty,(0,\beta_2))\\
&=\frac{N(\a_j)^{\frac12-s}}{[\Gamma_{1,j}(\q):H]}
\int_{H\backslash\GL_2(F_\infty)}
U_\xi\!\left(\begin{pmatrix}1&&\\&1&\beta_2\\&&1\end{pmatrix}
\begin{pmatrix}h&\\&1\end{pmatrix},
\begin{pmatrix}\aid_j&&\\&1&\\&&1\end{pmatrix}\right)
U_\varphi(h,h_j)\|\det{h}\|^{s-\frac12}\,dh.
\end{align*}
Let $\nid$ and $\qid$ be id\`eles such that $\n=(\nid)$ and $\q=(\qid)$.
We make the change of variables $h\mapsto{}^th^{-1}$ and use the
identities
$$
U_\xi(g)=\epsilon_\pi^2U_\xi'\!\left(
{}^tg^{-1}\begin{pmatrix}\nid&&\\&\nid&\\&&1\end{pmatrix}\right),
\quad
U_\varphi(g)=\epsilon_{\tau}U_\varphi'\!\left(
{}^tg^{-1}\begin{pmatrix}\qid&\\&1\end{pmatrix}\right),
$$
from Lemma~\ref{aut-dual} to get
\begin{align*}
\frac{\epsilon_\pi^2\epsilon_\tau N(\a_j)^{\frac12-s}}
{[\Gamma_{1,j}(\q):H]}
\int_{{}^tH\backslash\GL_2(F_\infty)}
&U_\xi'\!\left(\begin{pmatrix}1&&\\&1&\\&-\beta_2&1\end{pmatrix}
\begin{pmatrix}h&\\&1\end{pmatrix},
\begin{pmatrix}\aid_j^{-1}\nid&&\\&\nid&\\&&1\end{pmatrix}\right)\\
&\cdot U_\varphi'\!\left(h,
\begin{pmatrix}\aid_j^{-1}\qid&\\&1\end{pmatrix}\right)
\|\det{h}\|^{\frac12-s}\,dh.
\end{align*}
Next we make the change of variables $h\mapsto\gamma{h}$ for some
$\gamma\in\GL_2(F)$ and use automorphy of $U_\xi'$ and
$U_\varphi'$ to get
\begin{align*}
\frac{\epsilon_\pi^2\epsilon_\tau(N(\a_j)\|\det\gamma\|_\infty)^{\frac12-s}}
{[\Gamma_{1,j}(\q):H]}
\int_{\gamma^{-1}({}^tH)\gamma\backslash\GL_2(F_\infty)}
&U_\xi'\!\left(
\begin{pmatrix}h&\\&1\end{pmatrix},
\begin{pmatrix}\gamma^{-1}&\\&1\end{pmatrix}
\begin{pmatrix}1&&\\&1&\\&\beta_2&1\end{pmatrix}
\begin{pmatrix}\aid_j^{-1}\nid&&\\&\nid&\\&&1\end{pmatrix}\right)\\
&\cdot U_\varphi'\!\left(h,\gamma^{-1}
\begin{pmatrix}\aid_j^{-1}\qid&\\&1\end{pmatrix}\right)
\|\det{h}\|^{\frac12-s}\,dh.
\end{align*}

Put $\rid=\qid\beta_2$ and
choose $\gamma\in\GL_2(F)$, $\aid_k$ and $\kappa\in K(\q^2)$ such that
$\begin{psmallmatrix}\aid_j^{-1}\nid\qid&\\&\nid\qid^2\end{psmallmatrix}
=\gamma\begin{psmallmatrix}\aid_k&\\&1\end{psmallmatrix}\kappa$.
Then we have
$$
\begin{pmatrix}\gamma^{-1}&\\&1\end{pmatrix}
\begin{pmatrix}1&&\\&1&\\&\beta_2&1\end{pmatrix}
\begin{pmatrix}\aid_j^{-1}\nid&&\\&\nid&\\&&1\end{pmatrix}
=\qid^{-1}\begin{pmatrix}1&&-\beta_1'\\&1&-\beta_2'\\&&1\end{pmatrix}
\begin{pmatrix}\aid_k&&\\&1&\\&&1\end{pmatrix}
\begin{pmatrix}\kappa&\\&1\end{pmatrix}
\begin{pmatrix}1&\nid\rid\uid&\qid\uid\\&
\frac{\nid\rid\vid+1}{\qid}&\vid\\&\nid\rid&\qid\end{pmatrix},
$$
where $\begin{psmallmatrix}\beta_1'\\\beta_2'\end{psmallmatrix}
=\nid\qid\gamma^{-1}\begin{psmallmatrix}\aid_j^{-1}\uid\\\vid\end{psmallmatrix}
=\begin{psmallmatrix}\aid_k&\\&1\end{psmallmatrix}\kappa
\begin{psmallmatrix}\uid\\\qid^{-1}\vid\end{psmallmatrix}$.
We may choose $\uid,\vid\in\prod_\p\o_\p$ so that
$\frac{\nid\rid\vid+1}{\qid}\in\prod_\p\o_\p$, $(\nid\qid\vid)$ and
$(\nid\qid\aid_j^{-1}\uid)$ are
principal, and $\beta_1'=0$.
Taking determinants, we find that $(\det\gamma)=\n^2\q^3\a_j^{-1}\a_k^{-1}$.
Also,
$$
\gamma^{-1}\begin{pmatrix}\aid_j^{-1}\qid&\\&1\end{pmatrix}
=(\nid\qid^2)^{-1}\begin{pmatrix}\aid_k\qid^2&\\&1\end{pmatrix}
\cdot\begin{pmatrix}\qid^{-2}&\\&1\end{pmatrix}
\kappa\begin{pmatrix}\qid^2&\\&1\end{pmatrix}
\in(\nid\qid^2)^{-1}\begin{pmatrix}\aid_k\qid^2&\\&1\end{pmatrix}
K_1(\q^4).
$$

Hence, we obtain
\begin{align*}
\frac{cN(\n^2\q^3)^{\frac12-s}N(\a_k)^{s-\frac12}}
{[\Gamma_{1,j}(\q):H]}
\int_{\gamma^{-1}({}^tH)\gamma\backslash\GL_2(F_\infty)}
&U_\xi'\!\left(
\begin{pmatrix}1&&\\&1&\beta_2'\\&&1\end{pmatrix}
\begin{pmatrix}h&\\&1\end{pmatrix},
\begin{pmatrix}\aid_k&&\\&1&\\&&1\end{pmatrix}\right)\\
&\cdot U_\varphi'\!\left(h,
\begin{pmatrix}\aid_k\qid^2&\\&1\end{pmatrix}\right)
\|\det{h}\|^{\frac12-s}\,dh,
\end{align*}
where $c=\epsilon_\pi^2\epsilon_\tau\omega_{\widetilde{\pi}_f}(\qid^{-1})
\chi_{\widetilde{\pi}}(\qid)\omega_{\widetilde{\tau}_f}((\nid\qid^2)^{-1})
=\epsilon_\pi^2\epsilon_\tau\chi_{\omega_\pi}(\q)
\omega_{\tau_f}(\nid\qid^2)$.

We choose $H$ so that $\gamma^{-1}({}^tH)\gamma\subseteq\Gamma_k(\q^2)$.
Then $[\Gamma_{1,j}(\q):H]=[\Gamma_{1,k}(\q):\gamma^{-1}({}^tH)\gamma]$,
and by \eqref{genadd}, the above is
\begin{equation}\label{dualadd}
cN(\n^2\q^3)^{\frac12-s}
\Lambda_k(1-s,\widetilde{\pi},\widetilde{\tau},
\xi_\infty',\varphi_\infty',(0,\beta_2'),\qid^2),
\end{equation}
where $\xi_\infty'$ and $\varphi_\infty'$ are the vectors
in $V_{\widetilde{\pi}_\infty}$ and $V_{\widetilde{\tau}_\infty}$ such
that $W_{\xi_\infty'}=\widetilde{W}_{\tilde{\xi}_\infty}$ and
$W_{\varphi_\infty'}=\widetilde{W}_{\tilde{\varphi}_\infty}$.
Applying Theorem~\ref{main1} with $\q_1=\o_F$ and $\q_2=\q$, we have
$$
\sum_{j=1}^h\Lambda_j(s,\pi,\tau,\xi_\infty,\varphi_\infty,(0,\beta_2))
=\tau_\q(\chi_\tau,\beta_2)\Psi_\infty(s,\xi_\infty,\varphi_\infty)
L(s,\pi\times\tau).
$$
On the other hand, summing \eqref{dualadd} over $k$ and applying
\eqref{main1res} with $\zeta=\qid^2$ and $\pi,\tau,\psi_\infty$ replaced
by their duals, we get
\begin{align*}
&cN(\n^2\q^3)^{\frac12-s}
\overline{\tau_\q(\chi_\tau,\beta_2')}\lambda_{\widetilde{\tau}}(\q^2)
\Psi_\infty(1-s,\xi_\infty',\varphi_\infty')
L(1-s,\widetilde{\pi}\times\widetilde{\tau})\\
&=\epsilon N(\n^2\q^3)^{\frac12-s}
\widetilde{\Psi}_\infty(1-s,\xi_\infty,\varphi_\infty)
L(1-s,\widetilde{\pi}\times\widetilde{\tau}),
\end{align*}
where
\begin{equation}\label{epsdef}
\epsilon=\frac{\epsilon_\pi^2\epsilon_\tau\chi_{\omega_\pi}(\q)
\omega_{\tau_f}(\nid\qid^2)\lambda_{\widetilde{\tau}}(\qid^2)
\overline{\tau_\q(\chi_\tau,\beta_2')}}
{\tau_\q(\chi_\tau,\beta_2)}.
\end{equation}
Finally, since $\n+\q=\o_F$, we have
$L(s,\pi\boxtimes\tau)=L(s,\pi\times\tau)$ and
$L(s,\widetilde{\pi}\boxtimes\widetilde{\tau})
=L(s,\widetilde{\pi}\times\widetilde{\tau})$. This concludes
the proof.
\end{proof}

\section{$\GL_3\times\GL_1$}
In this section, $\tau=\omega$ is an id\`ele class character of $F$
of conductor $\q$, and $\pi$ is an irreducible admissible generic
representation of $\GL_3(\A_F)$. For each finite place $v$ of $F$,
the essential vector $\xi_v^0$ and the corresponding $W_{\xi_v^0}$ are
as described in the previous
section. For $\beta\in F^\times$, we embed $\beta$ in $\A_F$ via
$\beta\mapsto(\beta,0)\in F_\infty\times\A_{F,f}$ as before, and
for $\xi_\infty\in V_{\pi_\infty}$, $j=1,\ldots,h$, we define
the function $\Phi_{\xi_\infty,j}$ on $\A_F^\times$ via
\[
\Phi_{\xi_\infty,j}(y)=\PP_1^3(U_\xi)\!
\begin{pmatrix}(\beta y,\aid_j)&(\beta,0)&\\&1&\\&&1\end{pmatrix},
\quad\text{for }\xi=\xi_\infty\otimes\xi_f^0, y\in\A_F^\times.
\]
This function is not invariant under $F^\times$ due to the presence
of $(\beta,0)$, but as before it will follow from its Fourier expansion (see below) that it is invariant
under a suitable congruence subgroup when viewed as a function on
$F_\infty^\times$. We will from here onwards take $\beta\in\a\q^{-1}$, where $\a$ as before is the product $\prod_j\a_j$.

\subsection{The Fourier expansion of $\Phi_{\xi_\infty,j}$}
By definition, for $y\in F_\infty^\times$, we have
\[
\Phi_{\xi_\infty,j}(y)=\sum_{\gamma\in F^\times}
W_{\xi_\infty\otimes\xi_f^0}\!
\begin{pmatrix}\gamma(\beta y,\aid_j)&\gamma(\beta,0)&\\&1&\\&&1\end{pmatrix}.
\]
One checks that
\[
W_{\xi_f^0}\!\begin{pmatrix}\gamma\aid_j&&\\&1&\\&&1\end{pmatrix}
\ne0\implies\gamma\in \a_j^{-1}\cap F^\times,
\]
so that
\[
\Phi_{\xi_\infty,j}(y)=\sum_{\gamma\in\a_j^{-1}\cap F^\times}
W_{\xi_f^0}\!\begin{pmatrix}\gamma\aid_j&&\\&1&\\&&1\end{pmatrix}
W_{\xi_\infty}\!\begin{pmatrix}\gamma\beta y&\gamma\beta&\\&1&\\&&1\end{pmatrix}.
\]
Let us put
$a_{\xi_f^0}(\aid_j,\gamma)=W_{\xi_f^0}\!
\begin{psmallmatrix}\gamma\aid_j&&\\&1&\\&&1\end{psmallmatrix}$,
and note that it is invariant under $\gamma\mapsto\eta\gamma$ for
$\eta\in\o_F^\times$. Thus the above becomes
\[
\Phi_{\xi_\infty,j}(y)
=\sum_{\gamma\in\o_F^\times\backslash\a_j^{-1}\cap F^\times}
a_{\xi_f^0}(\aid_j,\gamma)\sum_{\eta\in\o_F^\times}\psi_\infty(\eta\gamma\beta)
W_{\xi_\infty}\!\begin{pmatrix}\eta\gamma\beta y&&\\&1&\\&&1\end{pmatrix}.
\]
Since $\gamma\beta\in\q^{-1}$, the map
$\eta\mapsto\psi_\infty(\eta\gamma\beta)$ factors through
$\Gamma_q=\{\epsilon\in\o_F^\times:\epsilon\equiv1\pmod*{\q}\}$,
and consequently
\begin{equation}\label{Phi2}
\Phi_{\xi_\infty,j}(y)
=\sum_{\gamma\in\o_F^\times\backslash\a_j^{-1}\cap F^\times}
a_{\xi_f^0}(\aid_j,\gamma)\sum_{\eta\in\Gamma_\q\backslash\o_F^\times}
\psi_\infty(\eta\gamma\beta)\sum_{\epsilon\in\Gamma_\q}
W_{\xi_\infty}\!\begin{pmatrix}\epsilon\eta\gamma\beta y&&\\&1&\\&&1\end{pmatrix}.
\end{equation}
In particular, it follows that $\Phi_{\xi_\infty,j}(y)$ is
$\Gamma_j(\q)$ invariant.

\subsection{Additive twists}
For $\beta\in\a_j\q^{-1}$, $\xi_\infty\in V_{\pi_\infty}$, and with the
rest of the notation as above, we define additive twist (the analogue
of \eqref{add}) in this situation to be
\[
\Lambda_j(s,\pi,\omega,\xi_\infty,\beta)
=\frac{N(\a_j)^{1-s}\|\beta\|_\infty^{s-1}}{[\o_F^\times:\Gamma_\q]}
\int_{\Gamma_\q\backslash F_\infty^\times}\Phi_{\xi_\infty,j}(y)
\omega_\infty(y)\|y\|_\infty^{s-1}d^\times y,
\]
where $\omega_\infty$ is the archimedean component of $\omega$. We
insert \eqref{Phi2} into the above expression and collapse the integral
and the sum over $\epsilon$ to get
\begin{align*}
\Lambda_j(s,\pi,\omega,\xi_\infty,\beta)
&=\frac{N(\a_j)^{1-s}\|\beta\|_\infty^{s-1}}{[\o_F^\times:\Gamma_\q]}
\sum_{\gamma\in\o_F^\times\backslash\a_j^{-1}\cap F^\times}
a_{\xi_f^0}(\aid_j,\gamma)\sum_{\eta\in\Gamma_\q\backslash\o_F^\times}
\psi_\infty(\eta\gamma\beta)\\
&\qquad\cdot\int_{F_\infty^\times}W_{\xi_\infty}\!
\begin{pmatrix}\eta\gamma\beta y&&\\&1&\\&&1\end{pmatrix}
\omega_\infty(y)\|y\|_\infty^{s-1}d^\times y.
\end{align*}
Changing $y\mapsto(\gamma\eta\beta)^{-1}y$, we obtain
\begin{align*}
\Lambda_j(s,\pi,\omega,\xi_\infty,\beta)&=N(\a_j)^{1-s}
\sum_{\gamma\in\o_F^\times\backslash\a_j^{-1}\cap F^\times}
\|\gamma\|_\infty^{1-s}a_{\xi_f^0}(\aid_j,\gamma)\\
&\quad\cdot\frac1{[\o_F^\times:\Gamma_\q]}
\sum_{\eta\in\Gamma_\q\backslash\o_F^\times}
\psi_\infty(\eta\gamma\beta)\omega_\infty(\eta\gamma\beta)^{-1}
\int_{F_\infty^\times}W_{\xi_\infty}\!
\begin{pmatrix}y&&\\&1&\\&&1\end{pmatrix}
\omega_\infty(y)\|y\|_\infty^{s-1}d^\times y.
\end{align*}
In the notation of \cite{B-Kr3}, the average over $\eta$
is $e_\q((\gamma\beta),\omega_\infty^{-1})$. Putting all of
this together we obtain
\[
\Lambda_j(s,\pi,\omega,\xi_\infty,\beta)
=\sum_{\gamma\in\o_F^\times\backslash\a_j^{-1}\cap F^\times}
N((\gamma)\a_j)^{1-s}a_{\xi_f^0}(\aid_j,\gamma)
e_\q((\gamma\beta),\omega_\infty^{-1})
\prod_{v\mid\infty}\Psi_v(s;W_{\xi_v},\omega_v).
\]
For $\gamma\in\o_F^\times\backslash\a_j^{-1}\cap F^\times$,
let $\a=(\gamma)\a_j$, then note that
\[
\lambda_\pi(\o_F,\a)=a_{\xi_f^0}(\aid_j,\gamma)N(\a).
\]
Consequently
\[
\Lambda_j(s,\pi,\omega,\xi_\infty,\beta)=\sum_{\a\sim\a_j}
\frac{\lambda_\pi(\o_F,\a)e_\q(\a\a_j^{-1}(\beta),\omega_\infty^{-1})}{N(\a)^s}
\cdot\prod_{v\mid\infty}\Psi_v(s;W_{\xi_v},\omega_v).
\]

\begin{lemma}\label{lem:standardcoeff}
Let $\lambda_\pi(\a)$ denote
the Dirichlet coefficients of $L(s,\pi)$, so that
\[
L(s,\pi\otimes\omega)=\sum_{\a\subset\o_F}
\frac{\lambda_\pi(\a)\chi_\omega(\a)}{N(\a)^s}
\]
for every unramified id\`ele class character $\omega$.
Then $\lambda_\pi(\a)=\lambda_\pi(\o_F,\a)$.
\end{lemma}
\begin{proof}
This  essentially follows from Lemma ~\ref{dd}, where we have
identified the double Dirichlet coefficient $\lambda_{\pi}(\m_1,\m_2)$ of $\pi$.
To be precise, by multiplicativity it suffices to verify the desired
identity at prime powers.
To that end, fix a prime $\p$ and let $v$ be the corresponding place of
$F$. Let $\alpha_1,\alpha_2,\alpha_3$ be the Langlands parameters of
$\pi_v$ and let $k\geq 0$ be any integer.
From the proof of Lemma~\ref{dd} we have
$\lambda_\pi(\o_F,\p^k)=s_{k,0,0}(\alpha_1,\alpha_2,\alpha_3)$.
On the other hand
\[
s_{k,0,0}(\alpha_1,\alpha_2,\alpha_3)=\sum_{m_1+m_2+m_3=k}\alpha_1^{m_1}\alpha_2^{m_2}\alpha_3^{m_3},
\]
where $m_1,m_2$ and $m_3$ are non-negative integers. The right-hand side
of this expression is precisely $\lambda_\pi(\p^{k})$ thus proving our
claim at prime powers.
\end{proof}

Let $L_j(s,\pi,\omega,\xi_\infty,\beta)$ denote the finite part of
$\Lambda_j(s,\pi,\omega,\xi_\infty,\beta)$. The following analogue of
Theorem~\ref{main1} for $\GL(3)\times\GL(1)$
is a consequence of \cite[Proposition 3.1]{B-Kr}.
\begin{corollary}
For any id\`ele class character character $\omega$, there are numbers
$\beta_i\in F$ and $c_{ij}\in\C$ (depending on $\omega$) such that
\[
L(s,\pi\times\omega)=\sum_{i,j}c_{ij}L_j(s,\pi,\omega,\xi_\infty,\beta_i).
\]
\end{corollary}
To end, we note the following corollary, which is another consequence
of identifying the Dirichlet coefficient of $L(s,\pi)$ in terms of
the associated essential function. Although this directly follows from
\eqref{unr3} by a local calculation (see \cite[Corollary 3.3]{Mat}),
the argument here is global in nature. Moreover, it can be extended
to $\GL(n)\times\GL(m)$ for arbitrary $n>m$, i.e.\ we can identify the
Dirichlet coefficients of $L(s,\pi\times\tau)$ for any pair $(\pi,\tau)$
in terms of the associated essential functions. We will investigate this
for $n>3$ in a forthcoming paper.
\begin{corollary}
Let $\xi\in V_\pi$ be a decomposable vector with $\xi_v=\xi_v^0$ for all
finite $v$. Then
\[
\int_{F^\times\backslash\A_F^\times}\PP_1^3(U_\xi)\!
\begin{pmatrix}h&&\\&1&\\&&1\end{pmatrix}
\|h\|^{s-1}\,d^\times h=L(s,\pi)\prod_{v\mid\infty}\Psi_v(s;W_{\xi_v}),
\]
where $\Psi_v(s;W_{\xi_v})=\int_{F_v^\times}W_{\xi_v}\!
\begin{psmallmatrix}a&&\\&1&\\&&1\end{psmallmatrix}\|a\|_v^{s-1}d^\times a$.
\end{corollary}
\begin{proof}
We have the Fourier expansion
$\PP_1^3(U_\xi)(g)=\sum_{\gamma\in F^\times}W_\xi\!\left(
\begin{psmallmatrix}\gamma&&\\&1&\\&&1\end{psmallmatrix}g\right)$.
Inserting this in the left-hand side above, we get
\[
\int_{F^\times\backslash\A_F^\times}\sum_\gamma W_\xi\!
\begin{pmatrix}\gamma h&&\\&1&\\&&1\end{pmatrix}\|h\|^{s-1}\,d^\times h.
\]
Using $\A_F^\times=\coprod_jF^\times F_\infty^\times
\aid_j(\prod_{v<\infty}\o_v^\times)$, we may write this as
\[
\sum_j\int_{\o^\times\backslash
F_\infty^\times}\sum_{\gamma\in\a_j^{-1}\cap F^\times}
W_{\xi_f^0}\!\begin{pmatrix}\gamma\aid_j&&\\&1&\\&&1\end{pmatrix}
W_{\xi_\infty}\!\begin{pmatrix}\gamma h_\infty&&\\&1&\\&&1\end{pmatrix}
\|\aid_j\|^{s-1}\|h\|_\infty^{s-1}\,d^\times h_\infty,
\]
which in turn is the same as
\[
\sum_j\int_{\o^\times\backslash F_\infty^\times}
\sum_{\gamma\in\o_F^\times\backslash\a_j^{-1}\cap F^\times}
a_{\xi_f^0}(\aid_j,\gamma)\sum_{\eta\in\o_F^\times}
W_{\xi_\infty}\!\begin{pmatrix}\gamma\eta h_\infty&&\\&1&\\&&1\end{pmatrix}
\|\aid_j\|^{s-1}\|h\|_\infty^{s-1}\,d^\times h_\infty.
\]
Thus the left-hand side becomes (by a similar calculation to the above)
\[
\sum_j
\sum_{\gamma\in\o_F^\times\backslash\a_j^{-1}\cap F^\times}
a_{\xi_f^0}(\aid_j,\gamma)\|\aid_j\|^{s-1}\|\gamma\|_\infty^{1-s}
\prod_{v\mid\infty}\Psi_v(s;W_{\xi_v}).
\]
Since $a_{\xi_f^0}(\aid_j,\gamma)=N(\a)^{-1}\lambda_\pi(\o_F,\a)$ for
$(\gamma)\a_j=\a$, by Lemma~\ref{lem:standardcoeff}, this is
\[
\sum_j\sum_{\a\sim\a_j}\frac{\lambda_\pi(\a)}{N(\a)^s}
\cdot\prod_{v\mid\infty}\Psi_v(s;W_{\xi_v}),
\]
proving our assertion.
\end{proof}

\end{document}